\documentclass[a4paper,12pt,fullpage]{article}

\usepackage{amssymb}
\usepackage{amsmath, amsthm}
\usepackage{epsfig}
\usepackage{xcolor}
\usepackage{ascmac} 
\usepackage{appendix}
\usepackage{todonotes}
\usepackage{hyperref}

\numberwithin{equation}{section}

\newtheorem{Lemma}{Lemma}[section]
\newtheorem{Cor}[Lemma]{Corollary}
\newtheorem{Th}[Lemma]{Theorem}
\newtheorem{Thm}{Theorem}
\newtheorem{Prop}[Lemma]{Proposition}
\newtheorem{Def}[Lemma]{Definition}
\newtheorem{Rk}[Lemma]{Remark}

\newcommand{\be}{\begin{equation}}
\newcommand{\ee}{\end{equation}}
\newcommand{\baa}{\begin{array}}
\newcommand{\eaa}{\end{array}}

\def\R{{\mathbb R}}
\def\N{{\mathbb N}}

\def\ep{\varepsilon}

\begin{document}

\title{\bf{Front propagation through a perforated wall}}
\author{Henri Berestycki$^{\hbox{\small{ a,b,c}}}$, 
Fran\c cois Hamel$^{\hbox{\small{ d}}}$ and 
Hiroshi Matano$^{\hbox{\small{ e}}}$\\
\date{}\\
\footnotesize{$^{\hbox{a }}$EHESS, CAMS, 54 boulevard Raspail, F-75006 Paris, France}\\
\footnotesize{$^{\hbox{b }}$ Department of Mathematics, University of Maryland, College Park, USA}\\
\footnotesize{$^{\hbox{c }}$ Senior Visiting Fellow, Institute of Advanced Study, HKUST, Hong Kong}\\
\footnotesize{$^{\hbox{d }}$Aix Marseille Univ, CNRS, I2M, Marseille, France}\\
\footnotesize{$^{\hbox{e }}$Meiji University, Meiji Institute for Advanced Study of Mathematical Sciences,}\\
\footnotesize{4-21-1 Nakano, Tokyo 164-8525, Japan}
\thanks{This work has received funding from the French ANR Project ANR-23-CE40-0023-01 ReaCh.  H.M is partially supported by KAKENHI (16H02151, 21H00995) }}
\maketitle

\begin{abstract}
We consider a bistable reaction-diffusion equation $u_t=\Delta u +f(u)$ on $\R^N$ in the presence of an obstacle $K$, which is a wall of infinite span with many holes. More precisely, $K$ is a closed subset of $\R^N$ with smooth boundary such that its projection onto the $x_1$-axis is bounded and that $\R^N \setminus K$ is connected. Our goal is to study what happens when a planar traveling front coming from $x_1 = -\infty$ meets the wall $K$.

We first show that there is clear dichotomy between \lq\lq propagation\rq\rq and ``blocking''.  In other words, the traveling front either passes through the wall and propagates toward $x_1=+\infty$ (propagation) or is trapped around the wall (blocking), and that there is no intermediate behavior. This dichotomy holds for any type of walls of finite thickness. Next we discuss sufficient conditions for blocking and propagation. For blocking, assuming either that $K$ is periodic in $y:=(x_2,\ldots, x_N)$ or that the holes are localized within a bounded area, we show that blocking occurs if the holes are sufficiently narrow. For propagation, three different types of sufficient conditions for propagation will be presented, namely ``walls with large holes'', ``small-capacity walls'', and ``parallel-blade walls''. We also discuss complete and incomplete invasions.
\end{abstract}

\tableofcontents


\section{Introduction}\label{s:introduction}

We consider a reaction-diffusion equation on $\R^N$, $N\geq 2$, in the presence of obstacles. The problem is formulated as follows:
\begin{equation}\label{E}
\left\{
\begin{array}{ll}
u_t=\Delta u +f(u), \ \ &x\in \Omega:=\R^N\setminus K,\\[4pt]
\dfrac{\partial u}{\partial \nu} = 0, \ \ &x\in \partial\Omega,
\end{array}\right.
\end{equation}
where $f\in C^1$ is a bistable nonlinearity satisfying, for some $0<\alpha<1$,
\begin{equation}\label{f}
f(0)=f(\alpha)=f(1)=0, \ \ f'(0)<0,\ f'(\alpha)>0, \ f'(1)<0, \ \ \ \int_0^1 f(s)ds>0,
\end{equation}
and the obstacle $K$ is a closed set with uniformly smooth boundary satisfying
\begin{equation}\label{K}
K\subset\{x\in\R^N\mid 0\leq x_1\leq M\}, \quad\ \Omega:=\R^N\setminus K\ \hbox{is connected},
\end{equation}
for some constant $M>0$. Here and in what follows we shall use the notation
\[
x=(x_1,x_2,\ldots,x_N)=(x_1,y),\quad y=(x_2,\ldots,x_N).
\]
By ``uniformly smooth'', we mean that there exists $\delta>0$ such that for every point $x^*\in\partial\Omega=\partial K$, the sets $\partial\Omega\cap\{|x-x^*|\leq \delta\}$ and $\Omega\cap\{|x-x^*|\leq \delta\}$ can respectively be expressed locally as a graph and a subgraph of a smooth function whose derivatives have uniform bounds that do not depend on $x^*$. We call $\nu$ the outward unit normal to $\Omega$ on $\partial\Omega$.

The condition in \eqref{f} guarantees that the one-dimensional equation
\[
u_t=u_{xx}+f(u) \ \ (x\in\R)
\]
possesses a traveling wave solution of the form $\phi(x-ct)$ where $c$ is a positive constant and the profile function $\phi$ satisfies
\begin{equation}\label{phi}
\left\{
\begin{array}{l}
\phi''(z)+c\hspace{0.5pt}\phi'(z)+f(\phi(z))=0\ \ (z\in\R),\\[6pt]
0<\phi<1,\ \ \phi(-\infty)=1,\ \phi(+\infty)=0.
\end{array}\right.
\end{equation}
It is known that $c$ is unique and the traveling wave profile $\phi$ is unique up to translation (\cite{FM}). Hereafter we set
\begin{equation}\label{phi(0)}
\phi(0)=\alpha.
\end{equation}
Then \eqref{phi} and \eqref{phi(0)} determine the function $\phi$ uniquely. Using the same function $\phi$, one can construct a special solution of the equation $u_t=\Delta u+f(u)$ on $\R^N$ of the form
\begin{equation}\label{planar}
u(t,x)=u(t,x_1,x_2,\ldots,x_N)=\phi(x_1-ct),
\end{equation}
which we call the {\it planar front solution}. The goal of the present paper is to study the behavior of the planar front solution in the presence of the wall $K$. 

To formulate this question more precisely, we have to first construct a solution of \eqref{E} that behaves like the planar front solution ~\eqref{planar} when $t$ is sufficiently negative and whose front approaches $K$ as time passes. The following theorem guarantees the existence of such a solution. Note that this theorem holds for any type of obstacle $K$ as long as it lies in the right half space $x_1\geq 0$. Therefore the thickness of $K$ need not be finite.

\begin{Thm}\label{th:ubar-phi}
Assume simply that $K\subset \{x\in\R^N\mid x_1\geq 0\}$ and that the boundary of $\Omega=\R^N\setminus K$ is uniformly smooth. Then here exists a unique entire solution $\bar u$ of~\eqref{E} satisfying $0<\bar u<1 \;(x\in\overline{\Omega}, t\in\R)$ and 
\begin{equation}\label{ubar-phi}
\lim_{t\to-\infty}\sup_{x\in\Omega} |\bar u(t,x)-\phi(x_1-ct)|=0. 
\end{equation}
This solution satisfies $\bar u_t>0$ for all $x\in\overline{\Omega}, \,t\in\R$.
\end{Thm}

\begin{Rk}\label{rk:Lemma1}
In our previous paper \cite[Theorem 2.1]{BHM1}, we stated basically the same result as above, except that we considered left-bound traveling waves of the form $\phi(x_1+ct)$ in \cite{BHM1}, while in the present paper we consider right-bound traveling waves that are given in the form $\phi(x_1-ct)$.  Apart from this sign difference, the analysis remains the same. However, as there was a small gap in the proof of \cite[Theorem 2.1]{BHM1} concerning the claim $u_t>0$, we give a complete proof of the theorem in Section \ref{s:ubar-phi}. 
\end{Rk}

Since the solution $\bar u$ is monotone increasing in $t$, the following limit exists, which we call {\it the limit profile}:
\begin{equation}\label{vbar}
\bar{v}(x):=\lim_{t\to+\infty} \bar u(t,x)\quad \hbox{(limit profile)}.
\end{equation}
This function $\bar{v}$ is a solution of the following stationary problem:
\begin{equation}\label{S}
\left\{
\begin{array}{ll}
\Delta v +f(v)=0, \ \ &x\in \Omega:=\R^N\setminus K \\[4pt]
\dfrac{\partial v}{\partial \nu} = 0, \ \ &x\in \partial\Omega.
\end{array}\right.
\end{equation}
The long-time behavior of the solution $\bar{u}$ can be understood from this limit profile $\bar{v}$.

Since $\bar{v}$ is the limit of $\bar{u}$ satisfying \eqref{ubar-phi} and $\bar{u}_t>0$, it clearly has the following property:
\begin{equation}\label{vbar-1}
0<\bar{v}\leq 1\ \ \hbox{in}\ \ \Omega,\quad\  \lim_{x_1\to -\infty}\bar{v}(x)=1.
\end{equation}
As we shall see later in Theorem \ref{th:dichotomy}, either of the following althernatives holds, which we call ``propagation'' and  ``blocking'', and there is no intermediate behavior:
\begin{equation}\label{pro-block}
\lim_{x_1\to +\infty} \bar{v}(x_1,y)=
\begin{cases}
\, 1\ \ & \hbox{\rm (propagation)},\\[2pt] 
\, 0\ \ &\hbox{\rm (blocking)}.
\end{cases}
\end{equation}
Furthermore, quite importantly, the above convergence is uniform regardless of the choice of the wall $K$ so long as it is confined in the region $\{0\leq x_1\leq M\}$. As a consequence of this uniformity, one can show that the limit of a sequence of walls that block the front is again a blocking wall (Corollary~\ref{cor:blocking}). In other words, the family of blocking walls is closed.

Note that the classification \eqref{pro-block} between propagation and blocking is defined for the limit profile of the special solution $\bar{u}$. One may then wonder what happens for other solutions. It turns out that a large class of solutions whose initial support is contained in the region $\{x_1\leq 0\}$ converge to the same limit $\bar{v}$ as $t\to+\infty$ (Theorem~\ref{th:general}). Therefore, the notion of propagation and blocking defined in \eqref{pro-block} has much broader relevance.

The organization of this paper is as follows. In Sections \ref{s:main0}, \ref{s:main-blocking}, \ref{s:main-propagation}, we present our main results. In Section \ref{s:main0}, we assume simply that $K$ satisfies \eqref{K} and prove Theorem~\ref{th:dichotomy}, which establishes the dichotomy~\eqref{pro-block}. We next show that many solutions with initial support in the region $\{x_1\leq 0\}$ converge to the same limit $\bar{v}$ defined in \eqref{vbar} as $t\to+\infty$ (Theorem \ref{th:general}).

In Section \ref{s:main-blocking}, we give sufficient conditions for blocking. More precisely, if all the holes are sufficiently narrow in a certain sense, then blocking occurs. For this result, we consider two cases: the case where the holes of $K$ are localized within a bounded area (Theorem~\ref{th:blocking-localized}) and the case where $K$ is periodic in $y\in\R^{N-1}_y$ (Theorem \ref{th:blocking-periodic}). 

For propagation, we give three different types of sufficient conditions, namely
\begin{itemize}
\setlength{\parskip}{0pt}
\setlength{\itemindent}{30pt}
\item[(a)] walls with large holes;
\item[(b)] small-capacity walls (or skeleton walls);
\item[(c)] parallel-blade walls. 
\end{itemize}

The case (a) is intuitively clear. If the wall has a large enough hole that allows a ball of radius $R_0>0$ to pass through from one side of the wall to the other side, where $R_0$ is a specific constant to be specified later, then propagation occurs (Theorem~\ref{th:large-holes}). 

The case (b) deals with walls that have small capacity. If the capacity of the wall is very small, then propagation occurs even if there is no large open space in the wall (Theorem~\ref{th:small-capacity}). This is typically the case when the wall is made of dense debris-like objects.

The case (c) deals with walls that consist of very thin pannels that are parallel to the $x_1$-axis. For example, in the case $N=2$, a parallel-blade wall consists of thin needle-like obstacles that are all parallel to the $x_1$-axis. Here again propagation can occur even if the space between the needles is narrow, so long as the needles are thin enough (Theorem~\ref{th:parallel}). 

Figure \ref{fig:propagation-walls} shows typical images of the above three types of walls.

\begin{figure}[h]
\begin{center}
\includegraphics[width=0.75\textwidth]
{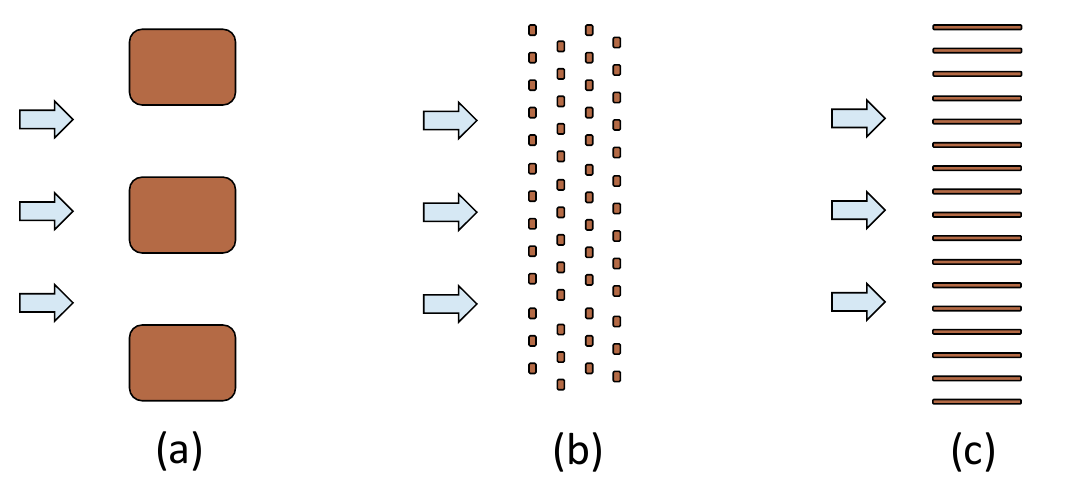}
\end{center}
\vspace{-10pt}
\caption{(a) wall with large holes; (b) small-capacity wall; (c) parallel-blade wall}
\label{fig:propagation-walls}
\end{figure}

We also present a result on {\it complete invasion}. More precisely we show that if the wall~$K$ satisfies certain geometrical conditions then we have $\bar v=1$ on entire $\Omega$ (Theorem~\ref{th:complete}).

In Section \ref{s:dichotomy-proof}, we prove the dichotomy theorem (Theorem \ref{th:dichotomy}), using a Liouville type theorem by Liu {\it et al} \cite{LWWW} (Theorem~\ref{th:Liouville}). 

In Section \ref{s:blocking}, we prove Theorems~\ref{th:blocking-localized} and \ref{th:blocking-periodic} on blocking.  The main idea is to construct an upper barrier, which is a stationary supersolution that tends to $0$ as $x_1\to +\infty$.  This barrier function is constructed by a variational method. 

In Section \ref{s:propagation} we prove the main results for propagation for the above three types of walls (a), (b), (c) (Theorems \ref{th:large-holes}, \ref{th:small-capacity}, \ref{th:parallel}).  The proof of Theorem~\ref{th:large-holes} for case (a) is rather straightforward, and it is based on the comparison principle and a sliding argument. In the proof of Theorem~\ref{th:small-capacity} for case (b), we combine the classical theory of removable singularities on a set of zero capacity and the above mentioned fact (Corollary~\ref{cor:blocking}) that the limit of a sequence of blocking walls is again blocking. 

The proof of Theorem~\ref{th:parallel} for case (c) is based on a rather non-standard comparison argument. More precisely, we construct a family of ``quasi-subsolutions'' $w^\lambda$ that slide along the $x_1$-axis and show that $w^\lambda$ remains ``nearly below'' the limit profle $\bar{v}$ for all $\lambda\in\R$. Here, $w^\lambda$ being nearly below $\bar{v}$ means that the measure of the set $\{x \mid w^\lambda(x) - \bar{v}(x)>0\}$ remains small. Once this is shown, we have $\bar{v}(x)\to 1$ as $x_1\to+\infty$. In order to make this argument work, we need a refined version of Poincar\'e inequality (Lemma~\ref{lem:poincare}), which will be proved in Appendix~\ref{s:poincare}. 

Section \ref{s:complete} is devoted to the proof of Theorems~\ref{th:complete} and \ref{th:incomplete} on complete and incomplete invasions. 


\section{Main results 1: dichotomy theorem}\label{s:main0}

In this section we consider general type of walls of finite thickness, that is, we simply assume the condition \eqref{K}, and state Theorem \ref{th:dichotomy} on the classification between propagation and blocking,  and Theorem \ref{th:general} on the behavior of solutions with compactly supported initial data.


\subsection{Dichotomy theorem}\label{ss:main-dichotomy}

We start with the results on the classification of the long-time behavior of the solution $\bar{u}$.

\begin{Thm}[Dichotomy]\label{th:dichotomy}
Assume \eqref{K}, and let $\bar{v}$ denote the limit profile defined in \eqref{vbar}. Then one of the following alternatives holds:
\[
\lim_{x_1\to +\infty} \bar{v}(x_1,y)=1\ \ \hbox{\rm (propagation)},\quad 
\lim_{x_1\to +\infty} \bar{v}(x_1,y)=0\ \ \hbox{\rm (blocking)} 
\]
Furthermore, the above convergence is uniform with respect to $y\in\R^{N-1}$ and $K$ so long as $K$ satisfies \eqref{K}. More precisely, for any $\ep>0$, there exists $M^\ep\geq M$ that does not depend on $K$ such that
\begin{equation}\label{vbar-ep}
\bar{v}(x_1,y)\in (0,\ep]\cup [1-\ep,1]\quad \hbox{for all}\ \ x_1\geq M^\ep,\ y\in\R^{N-1}.
\end{equation}
\end{Thm}

An immediate consequence of the above theorem is the following:

\begin{Cor}\label{cor:blocking}
Let $K_1, K_2, K_3,\ldots$ be a sequence of smooth walls satisfying 
\[
K_j\subset\{x\in\R^N\mid 0\leq x_1\leq M\}\ \  (j=1,2,3,\ldots)
\]
that converge to a wall $K_\infty$ in the Hausdorff distance. If blocking occurs for every $K_j\;(j=1,2,3,\ldots)$ then the same
holds for $K_\infty$. 
To be more precise, if $\bar v_j\;(j=1,2,3,\ldots)$ denote the limit profile for $K_j$, and if $\bar v_\infty$ denotes the limit of any convergent subsequence of $\{\bar v_j\}$, then 
\[
\lim_{x_1\to +\infty} \bar{v}_\infty(x_1,y)=0\quad \hbox{uniformly in}\ y\in\R^{N-1}.
\]
\end{Cor}

The above corollary implies that blocking walls form a closed family. This result will be exceedingly useful in the proof of propagation results for small capacity walls (Theorem~\ref{th:small-capacity}) and for parallel-blade walls (Theorem~\ref{th:parallel}) as we shall see in subsections \ref{ss:proof-small-capacity} and \ref{ss:proof-parallel}.

Somewhat similar dichotomy results are also known in other contexts. In~\cite{HZ}, the propagation of a solution $u(t,x)$ emanating from the planar front $\phi(x_1-ct)$ in a domain made up of, roughly speaking, a straight half-cylinder $\{x_1\le0,\,|y|<R\}$ and a cone $\{x_1\ge0,\,|y|<R+\beta|x_1|\}$ ($\beta>0$) was investigated: in that geometrical configuration, a dichotomy also holds for the limit profile $\bar{v}(x)$ of $u(t,x)$ as $t\to+\infty$, namely, either $\bar{v}=1$ in the domain (complete invasion), or $\bar{v}(x_1,y)\to0$ as $x_1\to+\infty$.


\subsection{Behavior of more general solutions}\label{ss:general}

So far, the notion of propagation and blocking has been defined by using the special solution~$\bar{u}$ satisfying \eqref{ubar-phi}. Here we consider the following initial-boundary value problem associated with \eqref{E} and show that many solutions of this problem share the same limit profile as $\bar{u}$, therefore the classification between propagation and blocking has much broader implications.
\begin{equation}\label{E2}
\left\{
\begin{array}{ll}
u_t=\Delta u +f(u), \ \ & t>0,\,x\in \Omega:=\R^N\setminus K,\\[4pt]
u(0,x)=u_0(x), \ \ &x\in \Omega,\\[3pt]
\dfrac{\partial u}{\partial \nu} = 0, \ \ & t>0,\,x\in \partial\Omega.
\end{array}\right.
\end{equation}

Before stating the theorem, we introduce some notation. Let $H(z),\,z\leq 0,$ be the function that is defined uniquely by the following conditions:
\[
H''+f(H)=0\ \ (-\infty<z<0), \quad H(0)=0,\quad \lim_{z\to-\infty}H(z)=1.
\]
Such a function exists since $f$ is an unbalanced bistable nonlinearity satisfying \eqref{f}. It is easily seen that $H'<0$ and that $0<H<1$ in $(-\infty,0)$. The function $H$ is extended by~$0$ in $(0,+\infty)$. Next let $\Psi^P(x)$ denote the compactly supported subsolution of \eqref{E} defined in~\eqref{Psi-P}. Then the following holds:

\begin{Thm}\label{th:general}
Let $u$ be a solution of \eqref{E2} whose initial data $u_0$ satisfies
\[
\Psi^P(x)\leq u_0(x)\leq H(x_1)\quad \hbox{for} \ \ x=(x_1,x_2,\ldots,x_N)\in \overline{\Omega}
\]
for some $P\in\Omega\cap\{x_1<0\}$. Then
\begin{equation}\label{limit-u}
\lim_{t\to+\infty} u(t,x)=\bar{v}(x),
\end{equation}
where $\bar{v}$ is the limit profile of the special solution $\bar{u}$ defined in \eqref{vbar}.
\end{Thm}

The proof of Theorems \ref{th:dichotomy} and \ref{th:general} will be given in Section \ref{s:dichotomy-proof}. More specifically, Theorem~\ref{th:dichotomy} will be proved by using a Liouville type result due to Liu {\it et al} \cite{LWWW}. Theorem~\ref{th:general} will be proved by first observing that $\bar{v}$ is the minimal among all the stationary solutions that satisfy \eqref{vbar-1} (Proposition~\ref{prop:minimal}). Once this minimality is established, Theorem \ref{th:general} follows immediately.


\section{Main results 2: sufficient conditions for blocking}\label{s:main-blocking}

In this section we discuss sufficient conditions for blocking (Theorems~\ref{th:blocking-localized} and \ref{th:blocking-periodic}). In addition to the finite-thickness condition \eqref{K}, we assume that either of the following holds:
\begin{itemize}
\item[{\bf (K1)}] (wall with localized holes): there exist $a,b$ with $0\leq a<b\leq M$ such that $\{x\in\R^N\mid a\leq x_1\leq b\}\setminus K$ is bounded;
\item[{\bf (K2)}] (periodic wall): there exist linearly independent vectors 
${\boldsymbol p}_2, \ldots, {\boldsymbol p}_{N}\in \R^{N-1}_y$ such that 
\begin{equation}\label{K-periodic}
K+{\boldsymbol p}_i=K\ \  (i=2,\ldots,N)\quad \hbox{(periodicity in $y$)}.
\end{equation}
\end{itemize}

Note that, in \eqref{K-periodic}, the vectors ${\boldsymbol p}_2, \ldots, {\boldsymbol p}_{N}\in \R^{N-1}_y$ are identified with those in $\R^N$ whose projection onto the $x_1$-axis is $0$.

For both (K1) and (K2), the blocking is proved by constructing a suitable upper barrier around the wall. The construction of the barrier is based on an variational argument. The proofs for the case (K1) and the case (K2) are essentially the same, except that the variational argument for (K2) is carried out on the unit periodicity cell defined in \eqref{unit-cell}.


\subsection{Blocking for walls with localized holes}\label{ss:main-blocking-localized}

In this subsection we consider the case (K1). We repeat our assumption:
\begin{equation}\label{K-single}
\{x\in\R^N\mid a\leq x_1\leq b\} \setminus K \ \ \hbox{is bounded for some $0\leq a<b\leq M$.}
\end{equation}
This includes the case where the wall $K$ has a single hole. We introduce some notation.
\[
\Omega_{b}:= \{ (x_{1}, y) \in \Omega\mid x_{1}> b \},\quad \Omega_{a, b}:= \{ (x_{1}, y) \in \Omega\mid a< x_{1} < b \}.
\]

\begin{Thm}[Blocking for walls with localized holes]\label{th:blocking-localized}
Assume that $\Omega_{b}$ is a uniformly Lipschitz domain. Then there exists $\ep>0$, depending on $f$, $b-a$ and $\Omega_b$, such that if~\eqref{K-single} holds and $|\Omega_{a,b}| \leq \ep$, then blocking necessarily occurs, where  $|A|$ denotes the Lebesgue measure of a set $A$.
\end{Thm}

It should be noted that the dependence of $\ep$ on the set $\Omega_b$ is quite subtle. In fact, if one fixes the passage $\Omega_{a, b}$, even a very narrow one, then it is shown in \cite{BBC} that an opening from this passage into the area $\{x_1\geq M\}$ that is {\em gradual} enough will allow the wave to propagate through the wall. Therefore, the narrowness of the passage $\Omega_{a,b}$ alone cannot guarantee blocking; whether blocking occurs or not depends on the combination of $\Omega_{a,b}$ and $\Omega_b$. 

We prove the above theorem in subsection \ref{ss:proof-general} by constructing a {\em barrier function} $w_{0}$ that is a stationary super solution of the elliptic equation in the region $\Omega_{-1}$ with $w_{0}(-1, y) = 1$ for all $y\in \R^{N-1}$ and $w_{0}(x_{1}, y) \to 0$ as $x_{1}\to \infty$. Such a barrier function blocks fronts as we have $\bar{u}(x,t)<w_{0}(x)$ for all $t\in\R,\,x\in\Omega_{-1}$.


\subsection{Blocking for periodic walls}
\label{ss:main-blocking-periodic}

Here we assume that $K$ satisfies \eqref{K-periodic}. In what follows, we shall denote the set of vectors ${\boldsymbol p}_2, \ldots, {\boldsymbol p}_{N}$ by ${\mathcal P}$. We say that a set $S\subset\R^N$ (or $S\in\R_y^{N-1}$) is {\bf ${\mathcal P}$-periodic} if 
\[
S=S+{\boldsymbol p}_2=\cdots=S+{\boldsymbol p}_{N}.
\]
Here the vectors ${\boldsymbol p}_2, \ldots, {\boldsymbol p}_{N}\in \R^{N-1}_y$ are identified with those in $\R^N$ whose projection onto the $x_1$-axis is $0$. Thus \eqref{K-periodic} means that $K$ is ${\mathcal P}$-periodic. We say that a function $w(x)=w(x_1,y)$ defined on a set $\Omega'\subset\R^N$ is {\bf ${\mathcal P}$-periodic} if its domain of definition $\Omega'$ is ${\mathcal P}$-periodic and if 
\[
w(x_1,y)=w(x_1,y+{\boldsymbol p}_2)=\cdots=w(x_1,y+{\boldsymbol p}_{N})\quad
\hbox{for all}\ (x_1,y)\in\Omega'.
\]
The {\it unit periodicity cell} associated with ${\mathcal P}$ is a set in $\R^{N-1}_y$ defined by
\begin{equation}\label{unit-cell}
{\mathcal C}_{\mathcal P}= \Big\{ \sum_{i= 2}^{N} t_i \,{\boldsymbol p}_i  \mid\,  0 < t_i< 1 \Big\}.
\end{equation}

\begin{Thm}[Blocking for periodic walls]\label{th:blocking-periodic}
Assume that $\Omega_{b}$ is a uniformly Lipschitz domain. Then there exists $\ep>0$, depending on $f$, $b-a$ and $\Omega_b$, such that if $|\Omega_{a,b} \cap \{ (a,b)\times {\mathcal C}_{\mathcal P} \}| \leq \ep$, then blocking necessarily occurs.
\end{Thm}

As in the previous subsection, we prove this theorem by constructing a {\em barrier function}~$w_{0}$ in the region $\Omega_{-1}$ satisfyin $w_{0}(-1, y) = 1$ for all $y\in \R^{N-1}$ and $w_{0}(x_{1}, y) \to 0$ as $x_{1}\to \infty$. As before, the construction of this barrier function is based on a variational argument, but this time, the variational argument is carried out on the unit periodicity cell ${\mathcal C}_{\mathcal P}$.


\section{Main results 3: sufficient conditions for propagation}\label{s:main-propagation}

Here we discuss sufficient conditions for propagation. As mentioned in Introduction, we present three different types of walls that allow propagation, namely: 
\[
\hbox{(a) walls with large holes, \quad
(b) small-capacity walls, \quad (c) parallel-blade walls.}
\]
We begin with the case (a).


\subsection{Conditions for propagation (a): walls with large holes}
\label{ss:main-large-holes}

This is a wall such that at least one of its holes is large enough to allow a ball of radius $R_0$ to pass through it.  Here the constant $R_0>0$ is defined as follows. Consider the problem
\begin{equation}\label{Psi0}
\begin{cases}
\,\Delta \Psi+f(\Psi)=0\ \ &(|x|<R),\\
\,\Psi =0\ \ &(|x|=R),\\
\,0<\Psi<1\ \ &(|x|<R).
\end{cases}
\end{equation}
This problem has a solution if $R>0$ is sufficiently large. To see this, consider the functional
\[
H[\Psi]=\int_{|x|\leq R}\left(\frac12 |\nabla \Psi|^2 - F(\Psi)\right)dx, \quad F(s):=\int_0^s f(\sigma)d\sigma
\]
under the boundary condition $\Psi(x)=0\,(|x|=R)$. We extend the domain of $F$ so that $F(s)<0$ for $s<0$ and $F(s)<F(1)$ for $s>1$. Since $F(1)=\max_{s\in\R} F(s)>F(0)=0$ by the assumption \eqref{f}, $H[\Psi]$ takes a negative value if $R$ is sufficiently large. The global minimizer of $H$ is therefore not $0$ for such $R$, and it is a solution of \eqref{Psi0}.  Furthermore, $\Psi$ is a radially symmetric decreasing function and satisfies
\begin{equation}\label{Psi(0)}
\alpha<\Psi(0)<1.
\end{equation}

Now we define
\begin{equation}\label{R0}
R_0 = \min\{R>0 \mid \hbox{\eqref{Psi0} has a solution} \}.
\end{equation}
The existence of the above minimum follows from standard elliptic estimates. The main result of this subsection is the following:

\begin{Thm}[Walls with large holes]\label{th:large-holes}
Suppose that there exists a continuous curve $\gamma$ connecting some point $P_1$ in the region $\{x_1<0\}$ and some point $P_2$ in the region $\{x_1>M\}$ such that the distance between $K$ and any point on $\gamma$ is larger than or equal to the constant~$R_0$ defined in \eqref{R0}. Then propagation occurs.
\end{Thm}


\subsection{Conditions for propagation (b): small-capacity walls}
\label{ss:main-small-capacity}

Our second type of wall is a wall of small capacity in a certain sense. Let us first recall the standard notion of zero capacity in $\R^N$.

\begin{Def}[Set of zero capacity]\label{def:zero-capacity}
Let ${\mathcal K}_0$ be a compact set in $\R^N$. We say that ${\mathcal K}_0$ has zero capacity if the following holds:
\begin{equation}\label{zero-capacity1}
\inf\left\{\int_{D} |\nabla w|^2 dx \;\Big |\; w\in C^1(\overline{D}),\ w\geq 1\;(x\in {\mathcal K}_0),\ w=0\;(x\in \partial D) \right\}=0,
\end{equation}
where $D$ is a bounded open set with smooth boundary containing ${\mathcal K}_0$. 
\end{Def}

As one can easily verify, the condition \eqref{zero-capacity1} depends only on ${\mathcal K}_0$ and does not depend on the choice of the open set $D\supset {\mathcal K}_0$. It is well-known that, if ${\mathcal K}_0$ is a set of zero capacity and if $D$ is an open set containing ${\mathcal K}_0$, then any bounded harmonic function defined on~$D\setminus {\mathcal K}_0$ can be extended to a harmonic function on $D$. Similarly, if $v$ is a bounded solution of $\Delta v+f(v)=0$ on $D\setminus {\mathcal K}_0$, then it can be extended to a solution of $\Delta v+f(v)=0$ on~$D$~(\cite{S}).

Here are two typical examples of situations in which a set has zero capacity (\cite{Po}).
\begin{itemize}
\item[(C1)] The $m$-dimensional Hausdorff measure of ${\mathcal K}_0$ is $0$ for any $m<N-2$;
\item[(C2)] ${\mathcal K}_0$ is a locally finite union of smooth $(N-2)$-dimensional manifolds.
\end{itemize}
For example, a discrete set has capacity $0$ for any $N\geq 2$. If $N=3$, a locally finite union of curves (allowing intersections) has zero capacity (the case (C2)).

Our main result of this type of walls is the following:

\begin{Thm}[Small capacity walls]\label{th:small-capacity}
Let $K^\ep\;(0<\ep\leq \ep_0)$ be a family of walls satisfying
\[
K^\ep\subset\{x\in\R^N\mid 0\leq x_1\leq M\}, \quad 
\limsup_{\ep\to 0}K^\ep\subset {\mathcal K}_1 \cup {\mathcal K}_0,
\]
where ${\mathcal K}_1$ is a closed set (possibly empty) satisfying the same condition as in Theorem \ref{th:large-holes} for $K={\mathcal K}_1$, while ${\mathcal K}_0$ is a closed set of zero capacity. 
Then for all sufficiently small $\ep>0$, propagation occurs for $K^\ep$.
\end{Thm}

\begin{figure}[h]
\begin{center}
\includegraphics[width=0.62\textwidth]
{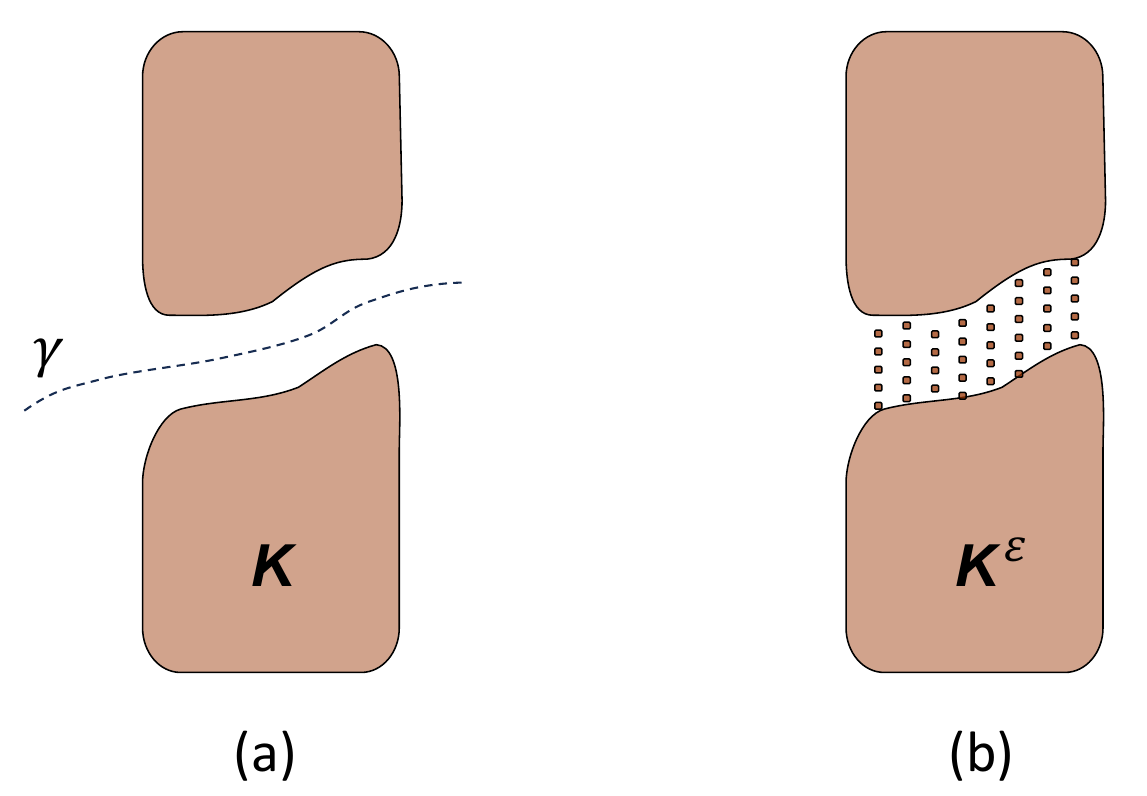}
\end{center}
\vspace{-10pt}
\caption{(a) wall with a large hole; \ \ (b) hole filled with small-capacity debris}
\label{fig:large-hole}
\end{figure}

Figure~\ref{fig:large-hole} (a) shows an example to which Theorem~\ref{th:large-holes} applies. The wall has a tunnel that allows a ball of radius $R_0$ to pass through. The dotted line indicates the curve $\gamma$.  Figure~\ref{fig:large-hole} exemplifies the situation to which Theorem~\ref{th:small-capacity} applies. Here, the same tunnel as in (a) is filled with debris-like objects. The front penetrates through the debris if its capacity is sufficiently small. In the special case where ${\mathcal K}_1=\emptyset$, every part of $K^\ep$ has small capacity.


\subsection{Conditions for propagation (c): parallel-blade walls}
\label{ss:main-parallel-blade}

Our third type of wall consists of objects that are very thin (like thin blades, possibly curved) that are all parallel to the $x_1$-axis. More precisely, it is defined as follows: 

\begin{itemize}
\item[(P)] Let $\Sigma\subset\R^{N-1}_y$ be a closed ${\mathcal P}$-periodic set formed by a locally finite union of smooth $(N-2)$-dimensional manifolds that intersect transversely with one another if they ever intersect, so that the boundary of each connected component of the set $\R^{N-1}_y\setminus \Sigma$ is a uniformly Lipschitz manifold.  For each small $\ep>0$, let ${\mathcal N}_\ep(\Sigma)$ denote the $\ep$-neighborhood of $\Sigma$. Finally, let $K^\ep$ $(0<\ep<\ep_0)$ be a family of ${\mathcal P}$-periodic closed sets in $\R^N$ with smooth boundary satisfying the following conditions:
\begin{equation}\label{P1}
K^\ep\subset [0,M]\times {\mathcal N}_\ep(\Sigma) \ \ \hbox{for all}\ \ 0<\ep\leq \ep_0, 
\end{equation}
\begin{equation}\label{P2}
\int_{\partial K^\ep\cap \Delta_{\mathcal P}} |\nu\cdot {\boldsymbol e}_1| \,dS_x\leq \ep_1  \ \ \hbox{for all}\ \ 0<\ep\leq \ep_0, 
\end{equation}
where $\nu$ denotes the outward unit vector to $\partial K^\ep$, $ {\boldsymbol e}_1$ the unit vector in the $x_1$ direction, $\Delta_{\mathcal P}:=\R\times {\mathcal C}_{\mathcal P}$ denotes the infinite cylinder in $\R^N$ whose cross section is ${\mathcal C}_{\mathcal P}$, the unit periodicity cell, 
and $\ep_1=\ep_1(\ep)$ is an $\ep$-dependent quantity that tends to $0$ as $\ep\to 0$.
\end{itemize}

The main theorem for this type of wall is the following:

\begin{Thm}[Parallel-blade walls]\label{th:parallel}
Let the family of ${\mathcal P}$-periodic obstacles $K^\ep$ satisfy \eqref{P1} and \eqref{P2}. Then for all sufficiently small $\ep>0$, propagation occurs for $K^\ep$.
\end{Thm}

\begin{Rk}\label{rk:blade}
To give the reader an idea of what the conditions \eqref{P1} means, let us consider the case $N=2$.  In this case, $\Sigma$ is a discrete set of points and $[0,M]\times\Sigma$ is a set of line segments that are all parallel to the $x_1$-axis and are aligned periodically in the $y$ direction.  The condition \eqref{P1} implies that $K^\ep$ consists of objects that are contained in an $\ep$ neighborhood of those line segments, therefore they are all thin objects of thickness at most $2\ep$. The condition~\eqref{P2} implies that the surface of those thin objects are rather flat in the middle part, while around  their edge the surface can have many tiny bumps so long as the total  length of the boundary around the edge remains small. In the case $N=3$, walls consisting of thin parallel panels and also those with honeycomb structure are typical examples of $K^\ep$.  
\end{Rk}

\begin{figure}[h]
\begin{center}
\includegraphics[width=0.82\textwidth]
{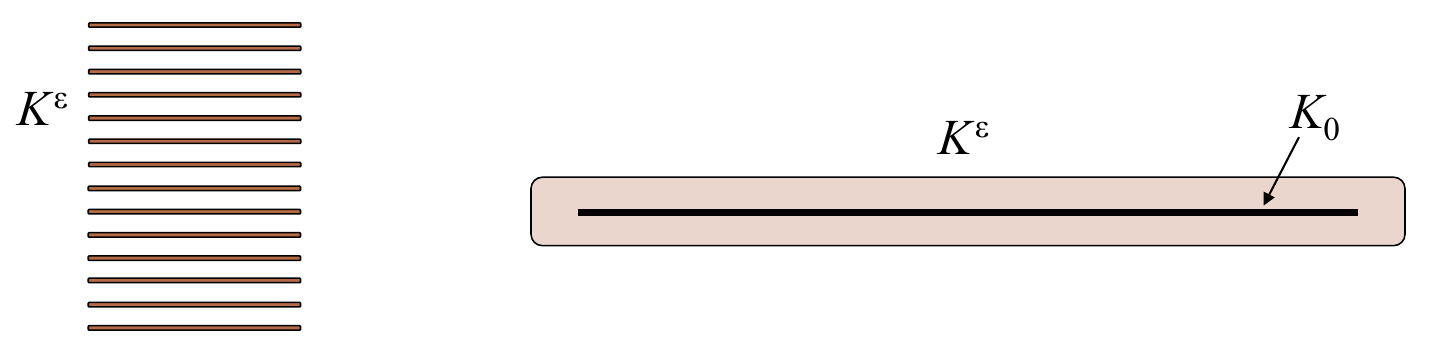}
\end{center}
\vspace{-10pt}
\caption{Image of a parallel-blade wall (left) and a magnified view of each blade (right). $K_0:=\lim_{\ep\to 0}K^\ep$ has positive $N-1$ Hausdorff measure, therefore its capacity is positive.}
\label{fig:parallel-blade}
\end{figure}


\subsection{Complete invasion}\label{ss:complete}

In this subsection we give a sufficient condition for the limit profile $\bar{v}$ to be identically equal to $1$, which we call {\it complete invasion}. 

\begin{Def}\label{def:directionary-convex}
$K$ is called {\it directionally convex} in the direction $x_1$ if, for some $a\in \R$, the following holds. Here ${\boldsymbol e}_1$ denotes the unit vector parallel to the $x_1$ axis.
\begin{itemize}\setlength{\parskip}{0pt}
\item[{\rm (i)}] for every line $\Lambda$ parallel to ${\boldsymbol e}_1$, the set $K\cap\Lambda$ is either a single line segment or empty;
\item[{\rm (ii)}] $K\cap \{x \in \R^N\mid x_1=a\}=\pi(K)$, where $\pi(K)$ is the orthogonal projection of $K$ onto the hypersurface $x_1=a$.
\end{itemize}
\end{Def}

Note that the above condition is slightly more stringent than the usual notion of ``directional convexity" because of the second condition $K\cap \{x\mid x_1=a\}=\pi(K)$. An example of directionally convex objects is given in Figure~\ref{fig:complete} (a). 

\begin{Thm}[Complete invasion]\label{th:complete}
Assume that $K$ is directionally convex in the direction $x_1$. If propagation occurs, then $\bar v(x) =1$ for all $x\in\Omega$. 
\end{Thm}

Essentially the same result is proved in our earlier paper \cite[Theorem 6.4]{BHM1}. Though the paper \cite{BHM1} dealt with the case where $K$ is a compact obstacle, the proof remains the same. The proof of Theorem~\ref{th:complete} will be given in Section~\ref{s:complete}. We think that complete invasion occurs for a much broader class of $K$ than just directional covexity. On the other hand, as  shown in \cite[Theorem 6.5]{BHM1}, complete invasion does not occur if part of $K$ has a reservoir-like shape with narrow entrance (see Figure \ref{fig:complete} (b)). Here we state this result in a somewhat vague manner. A more precise statement of this theorem is given in Theorem~\ref{th:incomplete2} in Section~\ref{s:complete}.

\begin{Thm}[Incomplete invasion]\label{th:incomplete}
Assume that part of $K$ has a reservoir-like configuration as shown in Figure \ref{fig:complete} (b). If the entrance of this reservoir is sufficiently narrow, then complete invasion does not occur. More precisely, the value of $\bar{v}$ remains close to $0$ inside the reservoir, even if propagation occurs. In other words, there are such cases that $0<\bar{v}<1$ on $\Omega$, while $\bar{v}(x)\to 1$ as $x_1\to \pm\infty$.
\end{Thm}

\begin{figure}[h]
\begin{center}
\includegraphics[width=0.6\textwidth]
{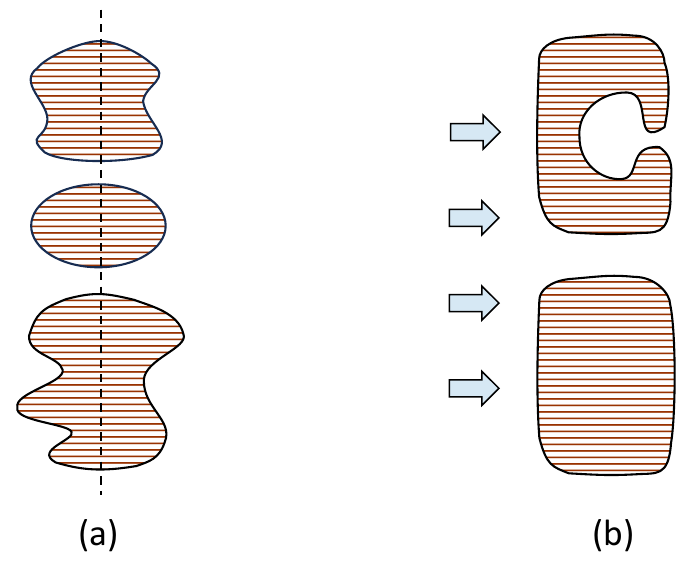}
\end{center}
\vspace{-10pt}
\caption{(a) Directionally convex objects; the dotted line indicates the hyperplane $x_1=a$. (b) A wall that has a reservoir with narrow entrance, which prevents complete invasion.}
\label{fig:complete}
\end{figure}

The above theorem can be proved by constructing an upper barrier around the mouth of the reservoir, in precisely the same style as in the proof of Theorem~\ref{th:blocking-periodic} for blocking. 


\section{Proof of Theorem \ref{th:ubar-phi}}\label{s:ubar-phi}

Here we prove Theorem \ref{th:ubar-phi} concerning the entire solution $\bar{u}$ satisfying \eqref{ubar-phi}. As we mentioned in Introduction, the same result appears in our ealier work \cite[Theorem 2.1]{BHM1}. The statement of \cite[Theorem 2.1]{BHM1} was completely correct, along with the proof of the existence of such an entire solution. However, there was a gap in the proof of the monotonicity of $\overline{u}(x,t)$ in $t$; see Remark~\ref{rk:historical} for details. In this section we give a complete proof of this theorem.  As we shall see, the monotonicity and uniqueness follows easily from the property \eqref{ubar-phi} (Proposition~\ref{prop:uniqueness}).


\subsection{Proof of the existence}\label{ss:existence}

The proof of the existence of $\bar{u}$ satisfying \eqref{ubar-phi} goes exactly along the same line as in the proof of Theorem 2.1 of \cite{BHM1}, except that $\varphi(x_1+ct)$ in \cite{BHM1} is replaced by $\phi(x_1-ct)$ here.  We prove the result under a slightly more general assumption that $f$ is a multistable nonlinearity that simply satisfies the following conditions:
\begin{equation}\label{f-multi}
\begin{cases}
\, f(0)=f(1)=0,\quad f'(0)<0,\ f'(1)<0,\\[2pt]
\, \hbox{there exists a function $\phi$ defined on $\R$ satisfying \eqref{phi} for some $c>0$.}
\end{cases}
\end{equation}
In other words, we assume that the one-dimensional equation $u_t=u_{xx}+f(u)$ possesses a traveling wave of speed $c>0$ connecting $0$ and $1$.

We construct the entire solution $\bar{u}$ as a limit of a sequence of solutions $(u_n)_{n\in\N}$ of the Cauchy problem \eqref{E2} that are defined for $-n\leq t<+\infty$ and are trapped between some sub- and super-solutions. Such an approach is found in \cite{HN1,HN2} for the construction of  new entire solutions of the Fisher-KPP equations, and also in \cite{FMN, GM} for the construction of entire solutions of a bistable reaction-diffusion equation having a pair of mutually annihilating fronts. For the supersolution, we rely in part on a technique of Guo and Morita \cite{GM}

\begin{proof}[Proof of the existence] 
As in~\cite{BHM1}, we introduce some auxiliary notations. Let
\[
\lambda=\frac{c+\sqrt{c^2-4f'(0)}}{2}
\]
be the positive root of the equation $\lambda^2-c\lambda+f'(0)=0$, let
\[
T=\frac{1}{\lambda\,c}\,\log\frac{c}{c+M_1}\ \in(-\infty,0)
\]
with $M_1>0$ being a free parameter to be chosen later, and let
\[
\xi(t)=\frac{1}{\lambda}\,\log\frac{c}{c-M_1\,e^{\lambda\, c\,t}}\,.
\]
The function $\xi$ is well defined in $(-\infty,T]$ and it solves the equation
\[
\xi'(t)=M_1\,e^{\lambda\,(ct+\xi(t))}\ \hbox{ in }(-\infty,T],\quad\hbox{with}\ \ \xi(-\infty)=0.
\]
Notice also that the function $t\mapsto ct+\xi(t)$ is increasing in $(-\infty,T]$ and that
\[
ct+\xi(t)\le cT+\xi(T)=0\ \hbox{ for all }t\le T.\]

Now let 
\[
H=\big\{x\in\R^N \mid x_1<0\big\}=\big\{x\in\Omega \mid x_1<0\big\}
\]
and define two functions $w^-$ and $w^+$ by
\begin{equation*}
w^-(t,x)=\left\{
\begin{array}{ll}\vspace{5pt}
\phi(x_1-ct+\xi(t))-\phi(-x_1-ct+\xi(t)) & \hbox{for }t\le T,\ x\in\overline{H},\\
0 & \hbox{for }t\le T,\ x\in\overline{\Omega}\setminus\overline{H},
\end{array}\right.
\end{equation*}
\begin{equation*}
w^+(t,x)=\left\{
\begin{array}{ll}\vspace{5pt}
\phi(x_1-ct-\xi(t))+\phi(-x_1-ct-\xi(t)) & \hbox{for }t\le T,\ x\in\overline{H},\\
2\,\phi(-ct-\xi(t)) & \hbox{for }t\le T,\ x\in\overline{\Omega}\setminus\overline{H}.
\end{array}\right.
\end{equation*}
Notice that $w^->0$ in $(-\infty,T]\times H$ since $\phi$ is decreasing, while $w^+>0$ in $(-\infty,T]\times\overline{\Omega}$, and that $w^{\pm}$ are both continuous in $(-\infty,T]\times\overline{\Omega}$. Furthermore, both functions $w^\pm$ are of class $C^2$ in $(-\infty,T]\times\big(\overline{\Omega}\!\setminus\!\{x_1=0\}\big)$, $w^+$ is of class~$C^1$ in $(-\infty,T]\times\overline{\Omega}$, and
\begin{equation}\label{w-Neumann}
\nu\cdot\nabla w^{\pm}=0\ \hbox{ on }(-\infty,T]\times\partial\Omega
\end{equation}
since $\partial\Omega\subset\R^N\!\setminus\!\overline{H}$ and $w^{\pm}$ do not depend on $x$ in $(-\infty,T]\times(\overline{\Omega}\!\setminus\!\overline{H})$. Note also that
\begin{equation}\label{w-<w+}
0\leq w^-(t,x)<w^+(t,x) \quad \hbox{on}\ \  (-\infty,T]\times\overline{\Omega} 
\end{equation}
and that
\begin{equation}\label{w-phi}
\lim_{t\to-\infty}\sup_{x\in \overline{\Omega}}|w^\pm(x,t)-\phi(x_1-ct\mp\xi(t))|=0.
\end{equation}

As shown in~\cite[Lemma~2.2]{BHM1}, if $M_1>0$ is chosen sufficiently large, there exists $T'$ with $T'\le T\,(\le 0)$ such that the following inequalities hold, the proof of which is omitted here:
\[
w^+_t\ge\Delta w^++f(w^+)\ \hbox{ in }(-\infty,T']\times\big(\overline{\Omega}\!\setminus\!\{x_1=0\}\big),
\]
\[
w^-_t\le\Delta w^-+f(w^-)\ \hbox{ in }(-\infty,T']\times\big(\overline{\Omega}\!\setminus\!\{x_1=0\}\big).
\]
We fix such $M_1$ and $T'\le0$ in the sequel. Since $w^-$ has a positive derivative gap at $x_1=0$, the above inequality together with \eqref{w-Neumann} implies that $w^-$ is a sub-solution of \eqref{E} in the time range $-\infty<t\leq T'$. Also, since $w^{+}$ has no derivative gap at $x_1=0$, the above inequality and \eqref{w-Neumann} imply that $w^+$ is a super-solution of \eqref{E} in the time range $-\infty<t\leq T'$. 

Now, for $n\ge-T'$, let $u_n(t,x)$ be the solution of \eqref{E} for $t\geq -n$ with initial data
\begin{equation}\label{un}
 u_n(-n,x)=w^-(-n,x).
\end{equation}
By \eqref{w-<w+}, we have $w^-(-n,x)=u_n(-n,x)<w^+(-n,x)$. Since $w^-$ (resp. $w^+$) is a sub- (resp. super-) solution, the comparison principle implies
\begin{equation}\label{w-un-w}
w^-(t,x)\leq u_n(t,x)\leq w^+(t,x)\quad\ \ \hbox{for}\ \ t\in[-n,T'],\ x\in\overline{\Omega}.
\end{equation}
Setting $t=-(n-1)$ in the above inequality yields, for $n\ge-T'+1$,
\[
u_{n}(-n+1,x)\geq w^-(-n+1,x)=u_{n-1}(-n+1,x).
\]
Applying again the comparison principle, we obtain
\begin{equation*}\label{un-monotone}
u_{n}(t,x)\geq u_{n-1}(t,x)\quad\ \ \hbox{for}\ \ t\in[-n+1,T'],\ x\in\overline{\Omega}.
\end{equation*}
Hence the sequence $u_n(t,x)$ is monotone increasing in $n$. Letting $n\to\infty$ and using parabolic estimates, we see that this sequence converges to an entire solution defined for $t\in\R,\,x\in\overline{\Omega}$, which we denote by $\bar u(t,x)$. Letting $n\to+\infty$ in \eqref{w-un-w} gives
\begin{equation*}\label{w-u-w}
w^-(t,x)\leq \bar u(t,x)\leq w^+(t,x)\quad\ \ \hbox{for}\ \
t\in(-\infty,T'],\ x\in\overline{\Omega}.
\end{equation*}
This, together with \eqref{w-phi}, and that fact that $\xi(t)\to 0$ as $t\to-\infty$ show that $\bar{u}$ satisfies \eqref{ubar-phi}. The proof of the existence is complete.
\end{proof}


\subsection{Proof of the uniqueness and monotonicity}\label{ss:comparison}

The uniqueness and time monotonicity of $\bar{u}$ in Theorem \ref{th:ubar-phi} are direct consequences of the property \eqref{ubar-phi}. More precisely, the following proposition holds:

\begin{Prop}\label{prop:uniqueness}
Let $\bar{u}$ be an entire solution of \eqref{E} satisfying \eqref{ubar-phi}. Then $\bar{u}_t>0$ for all $t\in\R$, $x\in\Omega$. Furthermore, there exists only one entire solution of \eqref{E} that satisfy \eqref{ubar-phi}.
\end{Prop}

This above proposition follows from Lemma \ref{lem:comparison-ancient} below. This lemma will also play an important role in the proof of Proposition \ref{prop:minimal} (the minimality of $\bar{v}$) and Theorem~\ref{th:large-holes}.

Before stating the lemma, we introduce some notation. Let $\delta_0\in (0,\frac12)$ be such that
\begin{equation}\label{delta}
f'(s)<0 \quad \hbox{for} \ \ s\in[0,\delta_0]\cup[1-\delta_0,1]. 
\end{equation}
As before, we use the notation $x=(x_1,x_2,\ldots,x_N)=(x_1,y)$, where $y=(x_2,\ldots,x_N)$.

\begin{Lemma}[comparison of ancient solutions]\label{lem:comparison-ancient}
Let $u(t,x), \widetilde{u}(t,x)$ be solutions of \eqref{E} defined on $(-\infty, T]\times\overline{\Omega}$ for some $T\in\R$ and satisfying $0\leq u\leq 1,\, 0\leq \widetilde{u}\leq 1$. 
 \begin{itemize}
\item[{\rm (i)}] 
Assume that there exist a smooth function $ a(t)$ such that
\[
u(t, x)<\widetilde{u}(t, x)\quad ({}\forall t\in (-\infty,T], \,x\in \Omega\cap\{x_1 = a(t)\}),
\]
\[
1-\delta_0\leq \widetilde{u}(t,x)\leq 1 \ \ \left({}\forall t\in (-\infty,T], \, x\in \Omega\cap\{x_1\leq  a(t)\}\right),
\]
Then
\[
u(t,x)<\widetilde{u}(t,x) \quad \hbox{for all}\ \ t\in (-\infty,T], \, x\in \Omega\cap\{x_1\leq  a(t)\}.
\]
\item[{\rm (ii)}] 
Assume that there exist a smooth function $b(t)$ such that
\[
u(t, x)<\widetilde{u}(t, x)\quad \left({}\forall t\in (-\infty,T], \,x\in \Omega\cap\{x_1=  b(t)\}\right),
\]
\[
0\leq u(t,x)\leq \delta_0 \ \ \left({}\forall t\in (-\infty,T],\,  x\in \Omega\cap\{x_1\geq  b(t)\}\right).
\]
Then
\[
u(t,x)<\widetilde{u}(t,x) \quad \hbox{for all}\ \ t\in (-\infty,T], \, x\in \Omega\cap\{x_1\geq  b(t)\}.
\]
\end{itemize}
\end{Lemma}

\begin{proof}
Since $f'<0$ on the compact set $[0,\delta_0]\cup[1-\delta_0,1]$, there exists  $\sigma>0$ such that
\[
f'(s)\leq -\sigma \quad \hbox{for} \ \ s\in[0,\delta_0]\cup[1-\delta_0,1]. 
\]
We first prove (i). Let $w(t,x):=u(t,x)-\widetilde{u}(t,x)$.  Then $w$ satisfies
\[
w_t=\Delta w+h(t,x)w \quad (t\in (-\infty,T],\, x\in \Omega\cap\{x_1\leq  a(t)\}),
\]
where
\[
h(t,x)=\int_0^1 f'\left(su(t,x)+(1-s) \widetilde{u}(t,x)\right) ds,
\]
along with the boundary condition
\begin{equation}\label{w(alpha)}
w(t, x)<0\quad (t\in(-\infty,T],\ x\in \Omega\cap\{x_1= a(t)\}).
\end{equation}
Also, since $1-\delta_0\leq \widetilde u\leq 1$ and $u\leq 1$, we have $w\leq \delta_0$. 
It suffices to show that $w<0$.

Suppose that $w(t,x)\geq 0$ for some $t\in(-\infty,T]$, $x\in \Omega\cap\{x_1\leq  a(t)\})$. Then we have $1-\delta_0\leq \widetilde{u}(t,x)\leq u(t,x)\leq 1$, which implies $h(t,x)\leq -\sigma$.  It follows that
\begin{equation}\label{w-sub}
w_t\leq\Delta w -\sigma w \quad \hbox{wherever}\ \ w\geq 0.
\end{equation}
Now we choose $T_1<T$ arbitrarily and define a function $\eta(t,x)=\delta_0 e^{-\sigma(t-T_1)}$. Then
\[
\eta_t=-\sigma \eta=\Delta\eta -\sigma \eta,\quad (t\in [T_1,T],\, x\in \Omega\cap\{x_1\leq  a(t)\}).
\]
Note also that
\[
\begin{split}
& \eta(T_1,x)=\delta_0 \geq w(T_1,x)\quad  (x\in \Omega\cap\{x_1\leq a(T_1)\}),\\
& \eta(t, a(t),y)>0>w(t, a(t),y)\ \ (t\in[T_1,T],\ (a(t),y)\in\Omega),\\
& \frac{\partial\eta}{\partial\nu}(t,x)=\frac{\partial w}{\partial\nu}(t,x)=0\quad (t\in[T_1,T],\ x\in\partial\Omega\cap\{x_1\le a(t)\}).
\end{split}
\]
Thus, in view of \eqref{w-sub}, $\eta$ acts as an upper barrier for $w$. Consequently,
\[
w(t,x)\leq \eta(t,x)=\delta_0 e^{-\sigma(t-T_1)}\quad (t\in [T_1,T],\, x\in \Omega\cap\{x_1\leq  a(t)\}).
\]
Recall that $T_1\in (-\infty,T)$ is arbitrary. Letting $T_1\to -\infty$, we obtain
\[
w(t,x)\leq 0 \quad \hbox{for all}\ \ t\in (-\infty,T],\, x\in \Omega\cap\{x_1\leq  a(t)\}.
\]
Since $w$ is not identically $0$ by virtue of \eqref{w(alpha)} ($\Omega\cap\{x_1=\xi\}\neq\emptyset$ for every $\xi\in\R$), the strong maximum principle implies $w<0$. The statement (i) is proved.

The statement (ii) is proved in the same manner, by simply replacing the region $\Omega\cap\{x_1\leq a(t)\}$ by $\Omega\cap\{x_1\geq  b(t)\}$. All we have to show is that $w=u-\widetilde{u}<0$ for $x_1\geq  b(t)$. Since $0\leq u\leq \delta_0$ and $0\leq \widetilde{u}\leq 1$, we have $0\leq \widetilde{u}\leq u\leq \delta_0$ whenever $w\geq 0$, which implies~\eqref{w-sub}. The conclusion $w<0$ then follows by arguing as above. The proof of Lemma~\ref{lem:comparison-ancient} is complete.
\end{proof}

Now we are ready to prove Proposition~\ref{prop:uniqueness}.

\begin{proof}[Proof of Propostion~\ref{prop:uniqueness}] 
We begin with the proof of $\bar{u}_t>0$. Let $\delta_0>0$ be as in Lemma~\ref{lem:comparison-ancient}. We recall that $\phi$ satisfies the condition \eqref{phi(0)}, that is, $\phi(0)=\alpha$. Let $L>0$ be such that
\[
1-\frac{\delta_0}{2}\leq \phi(z)<1 \ \ \hbox{for}\ \ z\in(-\infty, -L],\quad 
0<\phi(z)\leq \frac{\delta_0}{2}\ \ \hbox{for}\ \ z\in [L,+\infty)
\]
and define $a(t)=ct-L$, $b(t)=ct+L$. Then
\[
1-\frac{\delta_0}{2}\leq\phi(x_1-ct)< 1\ \ \hbox{if}\ \ x_1\leq a(t),\quad 
 0<\phi(x_1-ct)\leq \frac{\delta_0}{2}\ \ \hbox{if}\ \ x_1\geq b(t).
\]
Now let $\tau>0$ be a constant. Then since $\phi$ is monotone decreasing, we have
\[
\phi(x_1-c(t+\tau))-\phi(x_1-ct)>0.
\]
Since $\bar{u}(t,x)$ converges to $\phi(x_1-ct)$ and $\bar{u}(t+\tau,x)$ to $\phi(x_1-c(t+\tau))$ as $t\to-\infty$ uniformly on $\overline{\Omega}$, we see that, for any sufficiently small $\tau>0$, there exists $T<0$ such that
\begin{equation}\label{ubar-ubar-a}
1-\delta_0 \leq \bar{u}(t,x),\,\bar{u}(t+\tau,x)<1\quad \hbox{for}\ \ t\in(-\infty,T],\, x\in \overline{\Omega}\cap\{x_1\leq a(t)\},
\end{equation}
\begin{equation}\label{ubar-ubar-b}
0< \bar{u}(t,x),\,\bar{u}(t+\tau,x)\leq\delta_0 \quad \hbox{for}\ \ t\in(-\infty,T],\, x\in \overline{\Omega}\cap\{x_1\geq b(t)\},
\end{equation}
\begin{equation}\label{ubar-ubar-ab}
\bar{u}(t+\tau,x)-\bar{u}(t,x)>0 \quad \hbox{for}\ \ t\in(-\infty,T],\, x\in \overline{\Omega}\cap\{a(t)\leq x_1\leq b(t)\}.
\end{equation}
Combining \eqref{ubar-ubar-a}, \eqref{ubar-ubar-b}, $\bar{u}(t+\tau,a(t))-\bar{u}(t,a(t))>0$, $\bar{u}(t+\tau,b(t))-\bar{u}(t,b(t))>0$ and applying Lemma~\ref{lem:comparison-ancient} with $\widetilde{u}(t,x):=\bar{u}(t+\tau,x)$ and $u(t,x):=\bar{u}(t,x)$, we see that $\bar{u}(t+\tau,x)-\bar{u}(t,x)>0$ if $t\leq T$ and if $x_1\leq a(t)$ or $x_1\geq b(t)$. This, together with \eqref{ubar-ubar-ab}, imply
\[
\bar{u}(t+\tau,x)-\bar{u}(t,x)>0\quad\ \hbox{for}\ \ t\in(-\infty,T],\,x\in\overline{\Omega}. 
\]
By the comparison principle, the same inequality holds for $t\geq T$. Therefore $\bar{u}(t+\tau,x)-\bar{u}(t,x)>0$ for all $t\in\R$ and $x\in\overline{\Omega}$. Consequently
\[
\bar{u}_t(t,x)=\lim_{\tau\to+0}\frac{\bar{u}(t+\tau,x)-\bar{u}(t,x)}{\tau}\geq 0.
\]
Since $\bar{u}_t$ is not identically $0$, the strong maximum principle implies $\bar{u}_t>0$.

Next we prove the uniqueness. Suppose that $\widehat{u}$ satisfies 
\[
\lim_{t\to-\infty}\sup_{x\in\Omega} |\widehat{u}(t,x)-\phi(x_1-ct)|=0. 
\]
Then, by setting $\widetilde{u}(t,x)=\widehat{u}(t+\tau,x)$, $u(t,x)=\bar{u}(t,x)$ and applying Lemma~\ref{lem:comparison-ancient} as above, we see that $\widehat{u}(t+\tau,x)>\bar{u}(t,x)$ on $\R\times\overline{\Omega}$ for all sufficiently small $\tau>0$. Letting $\tau\to 0$, we obtain $\widehat{u}\geq \bar{u}$. Similarly, by setting $\widetilde{u}(t,x)=\bar{u}(t+\tau,x)$, $u(t,x)=\widehat{u}(t,x)$, we obtain $\bar{u}\geq\widehat{u}$. Thus we have $\widehat{u}=\bar{u}$ and the proof of the uniqueness is complete.
\end{proof}

\begin{Rk}
\label{rk:historical}
As mentioned earlier,
Theorem~\ref{th:ubar-phi} of the present paper states the same result as Theorem 2.1 of our earlier paper \cite{BHM1}. However, there was a gap there in the part concerned with the proof of the monotonicity of $\bar{u}$ in $t$ which was pointed out to us by S.~Eberle. The gap in \cite{BHM1} lies in the  claim that the subsolution $w^-(t,x)$ used in the construction of the entire solution $\bar{u}$ is monotone increasing in $t$. This is not true and we cannot infer from it, as in \cite{BHM1}, that the function $u_n$ defined by \eqref{un} satisfies $(u_n)_t>0$.

Yet, the monotonicity property of the limit function $\bar{u}$ holds true. There are different ways to fill the gap in proving it. The quickest way is to simply modify the definition of $u_n$ in \eqref{un}, i.e. $u_n(-n,x)=w^-(t,x)$, replacing it with
\begin{equation}\label{un-2}
u_n(-n,x)=\sup_{s\leq -n} w^-(s,x).
\end{equation}
Then it is easily seen that $u_n$ satisfies $(u_n)_t>0$ and $w^-\leq u_n\leq w^+$ as desired. The key-point of this corrected proof was indicated privately by one of the authors to S.~Eberle and was used in his paper \cite{E} for the construction of front-like entire solutions of some heterogeneous bistable reaction-diffusion equations in straight infinite cylinders. As a matter of fact, the same idea of defining $u_n$ as in \eqref{un-2} to construct a monotone increasing entire solution is also found in the proof of \cite[Theorem 5]{M1} for constructing a monotone increasing orbit emanating from an unstable equilibrium point in an order-preserving dynamical system. 

In the present paper, we have taken a different and new approach, which is to keep the existence proof in \cite{BHM1} as it is and to derive the monotonicity of $\bar{u}$ directly from the property~\eqref{ubar-phi}. This way, we can prove that monotonicity holds in general as a consequence of property~\eqref{ubar-phi}. Another advantage of this approach is that, once Lemma~\ref{lem:comparison-ancient} is established, the uniqueness and the monotonicity can be derived simultaneously. Furthermore, Lemma~\ref{lem:comparison-ancient} turns out to be a powerful tool. In fact,  it will also play an important role here in the proof of the minimality of $\bar{v}$ in Proposition~\ref{prop:minimal} as well as for the blocking results of Section~\ref{s:blocking}.

The method of proof we introduce here is of independent interest and could be applied in other situations as well.
\end{Rk}


\section{Proof of the dichotomy theorem}\label{s:dichotomy-proof}

In this section we prove Theorem~\ref{th:dichotomy} (dichotomy theorem) and Theorem~\ref{th:general} (universality of the limit profile $\bar{v}$). Throughout this section, we only assume that $K$ satisfies \eqref{K} (finite thickness) besides the smoothness of $\partial K$.


\subsection{A Liouville type result}

Before proving Theorem~\ref{th:dichotomy}, we recall a recent result by Liu {\it et al} \cite{LWWW} on a Liouville type theorem. Let $v$ be a solution of the following equation on the entire space $\R^N$:
\begin{equation}\label{semil}
\Delta v + g(v)=0 \ \ \ \hbox{in} \ \ \R^N
\end{equation}
where $g:\R\to\R$ is a $C^1$ function whose zeros are all isolated. In order to define the stability of the solution $v$, we consider linearized eigenvalue problems of the following form in balls~$B_{R}$ of radius $R$ with the Dirichlet boundary conditions:
\be\label{linearized}
\left\{
\begin{array}{ll}
- \Delta \phi -  g'(v) \phi = \lambda_{R} \phi , \ \ &\text{with} \ \ \phi>0 \ \ \text{in} \ \ B_{R} , \\[4pt]
\phi = 0 \ \ &\text{on} \ \ \partial B_{R}.
\end{array}\right. 
\ee

\begin{Def}\label{def:stability}
We say that a solution $v$ of \eqref{semil} is stable if the principal eigenvalue $\lambda= \lambda_{R}$ of the problem \eqref{linearized} satisfies $\lambda_{R} \geq 0$ for all $R>0$.
\end{Def}

Notice that this notion of stability is defined in a weak sense as it allows $\lambda_R$ to be $0$.  

\begin{Th}\label{th:Liouville}{\rm{(\cite[Theorem 1.4]{LWWW})}}
Let $v$ be a bounded solution of 
\[
\Delta v + g(v)=0 \ \ \ \hbox{in} \ \ \R^N
\]
that is stable in the sense of Definition~\ref{def:stability}. Assume that the one-dimensional equation
\[
w''+g(w)=0 \quad\hbox{in}\ \ \R
\]
does not have any nonconstant bounded stable solution. Then $v$ is a constant.
\end{Th}

Since our nonlinearity $f$ (extended by $f(s)=f'(0)s$ for $s<0$ and $f(s)=f'(1)(s-1)$ for $s>1$) clearly satisfies the assumption of the above theorem, we have the following corollary:

\begin{Cor}\label{cor:Liouville}
Let $f$ be as in \eqref{f} and let $v$ be a bounded solution of 
\begin{equation}\label{E-RN}
\Delta v+f(v)=0\quad \hbox{in}\ \ \R^N
\end{equation}
that is stable in the sense of Definition~\ref{def:stability} with $g=f$. Then either $v=0$ or $v=1$.
\end{Cor}

\begin{proof}
By Theorem~\ref{th:Liouville}, $v$ is a constant. Therefore $v$ is either $0$ or $\alpha$ or $1$. By the stability assumption, $v=\alpha$ is excluded.
\end{proof}


\subsection{Proof of Theorem \ref{th:dichotomy}}\label{ss:proof-dichotomy}

\begin{proof}[Proof of Theorem \ref{th:dichotomy}]
Suppose that \eqref{vbar-ep} does not hold for some $\ep=\ep_0>0$.  Then there exists a sequence of walls 
\[
K_j\subset\{x\in\R^N\mid 0\leq x_1\leq M\}\ \  (j=1,2,3,\ldots)
\]
and a sequence of real numbers $b_j\to \infty$ such that the limit profile $v_j$ corresponding to the wall $K_j$ satisfies
\[
\ep_0 < v_j(b_j,y_j) < 1-\ep_0 \quad (j=1,2,3,\ldots)
\]
for some $y_j\in\R^{N-1}$. Define a function $w_j(x)=w_j(x_1,y)$ on $\{(x_1,y)\mid M-b_j< x_1<\infty,y\in\R^{N-1}\}$ by $w_j(x_1,y)=v_j(x+b_j,y+y_j)$.  Then $w_j$ satisfies
\[
\Delta w_j + f(w_j)=0 \ \ \left((x_1,y)\in(M-b_j,\infty)\times\R^{N-1}\right), \quad 
\ep_0<w_j(0,0)<1-\ep_0.
\]
Since $w_1, w_2, w_3,\ldots$ are uniformly bounded, we can choose a subsequence of $(w_j)_{j\in\N}$ that converges locally uniformly in the $C^2$ sense to a function $w_\infty$ on $\R^N$ satisfying
\[
\Delta w_\infty + f(w_\infty)=0\ \ \ (x\in\R^N)
\]
along with the inequality
\begin{equation}\label{w-infty-ineq}
\ep_0\leq w_\infty(0,0)\leq 1-\ep_0.
\end{equation}
Recall that each $v_j$ is stable from below since it is a limit of a sequence of increasing solutions $\bar{u}_j(t,x)$ as $t\to+\infty$. Therefore $v_j$ is stable in the sense of Definition~\ref{def:stability}. Since such stability is robust under spatial translation and limiting procedures, we see that $w_\infty$ is also stable in the same sense.  Consequently, by Corollary~\ref{cor:Liouville}, we have either $w_\infty =0$ or $w_\infty=1$, but this contradicts the inequality \eqref{w-infty-ineq}. This contradiction proves Theorem~\ref{th:dichotomy}.
\end{proof}


\subsection{Proof of Theorem \ref{th:general}}\label{ss:proof-general}

We begin with the following proposition which states that the limit profile $\bar{v}$ defined in \eqref{vbar} is the minimal among all stationary solutions satisfying \eqref{vbar-1}.

\begin{Prop}[Minimality]\label{prop:minimal}
Let $v$ be a solution of the stationary problem \eqref{S} such that 
\[
0<v\leq 1\ \ \hbox{in}\ \ \Omega,\quad\  \lim_{x_1\to -\infty}v(x)=1.
\]
Then $v\geq \bar{v}$ on $\Omega$.
\end{Prop}

\begin{proof}
Let $\delta_0\in(0,\frac12)$ be the constant that appears in \eqref{delta}, and let $a<0$ be such that
\[
1-\delta_0 \leq v(x) \leq 1 \quad \hbox{for all}\ \ x\in\R^N\ \ \hbox{with}\ \ x_1\leq a.
\]
Next choose $T<0$ sufficiently negative so that
\[
0<\bar{u}(t,x)\leq \delta_0 \quad \hbox{for all}\ \ t\in (-\infty,T],\ x\in\Omega \ \ \hbox{with}\ \ x_1\geq a.
\]
Such $T$ exists since $\bar{u}$ satisfies \eqref{ubar-phi}. In particular, we have
\[
\bar{u}(t,a,y)<v(a,y) \quad \hbox{for all}\ \ t\in (-\infty,T].\ y\in\R^{N-1}.
\]
Applying Lemma~\ref{lem:comparison-ancient} (i) in the region $x_1\leq a$ and (ii) in the region $x_1\geq a$, we see that 
\[
\bar{u}(t,x)<v(x)\quad\hbox{for all}\ \ t\in (-\infty,T],\ x\in\Omega.
\]
By the comparison theorem, the above inequality holds also for $t\geq T$, hence $\bar{u}<v$ everywhere. Consequently, $\bar{v}=\lim_{t\to+\infty}\bar{u}\leq v$. The proposition is proved.
\end{proof}

\begin{proof}[Proof of Theorem~\ref{th:general}]
We first show that $H<\bar{v}$ in the region $x_1\leq 0$. Let $\delta_0>0$ be as in~\eqref{delta} and choose $a<0$ such that
\[
1-\delta_0 \leq \bar{v}(x)<1 \quad\hbox{for all}\ \ x\in\Omega\ \ \hbox{with}\ \ x_1\leq a.
\]
For each $\lambda\leq 0$, define $H^\lambda(x)=H^\lambda(x_1,y)=H(x_1-\lambda,y)$. Then $H^\lambda$ is defined in the region $x_1 \leq \lambda$ and satisfies $0\leq H^\lambda<1$. Note also that
\[
H^a(a,y)=0<\bar{v}(a,y) \quad\hbox{for all}\ \ y\in\R^{N-1}.
\]
Applying Lemma~\ref{lem:comparison-ancient} to $\tilde{u}=\bar{v}$ and $u=H^a$, we see that $H^a<\bar{v}$ in the region $x_1\leq a$. Now we let $\lambda$ vary from $a$ to $0$ continuously. Then, by the strong maximum principle, the graph of $H^\lambda$ remains strictly under that of $\bar{v}$ as $\lambda$ varies from $a$ to $0$. Consequently, 
\[
H(x)=H^0(x)<\bar{v}(x) \quad\hbox{for all}\ \ x\in\Omega\ \ \hbox{with}\ \ x_1\leq 0.
\]

Next let $U^P$ be the solution of \eqref{E2} whose initial data is $\Psi^P(x):=\Psi(x-P)$. Since $\Psi^P$ is a subsolution, $U^P(t,x)$ is monotone increasing in $t$. This, together with the fact that $\Psi^P\leq H<\bar{v}$, we see that $U^P$ converges as $t\to+\infty$ to some positive solution $V^P$ of \eqref{S} satisfying
\begin{equation}\label{VP<vbar}
\Psi^P<V^P \leq \bar{v} \quad \hbox{on} \ \ \Omega.
\end{equation}

Now we choose an arbitrary point $Q$ in $\Omega\cap\{x_1\leq -R_0\}$, where $R_0$ is the constant defined in \eqref{R0}, that is, the radius of the support of $\Psi$ and $\Psi^P$. For each $s\in[0,1]$, let $P(s)=(1-s)P+sQ$ be the interpolation point between $P$ and $Q$, and consider the family of subsolutions $\{\Psi^{P(s)}\}_{s\in[0,1]}$. At $s=0$, we have $\Psi^{P(0)}<V^P$ by \eqref{VP<vbar}. As $s$ increases continuously from $0$ to $1$, $\Psi^{P(s)}$ remains strictly below $V^P$ since otherwise the graph of $\Psi^{P(s^*)}$ touches that of $V^P$ from below for some $s=s^*$, but this is impossible by the strong maximum principle. Hence $\Psi^Q<V^P$ for all $Q$ in $\Omega\cap\{x_1\leq -R_0\}$, which implies
\[
V^P(x_1,y)>\Psi(0)>\alpha \quad \hbox{for all}\ \ x_1\leq -R_0,\ y\in\R^{N-1}
\]
since $\max_x \Psi^Q(x)=\max_x\Psi(x-Q)=\Psi(0)>\alpha$; see \eqref{Psi0} and \eqref{Psi(0)}. Since $f(s)>0$ for $\alpha<s<1$, the above inequality and a comparison argument imply
\[
\lim_{x_1\to-\infty}V^P(x_1,y)=1.
\]
Hence, by Proposition~\ref{prop:minimal}, $V^P\geq \bar{v}$. Therefore $V^P=\bar{v}$ by \eqref{VP<vbar} . In other words,
\[
\lim_{t\to\infty}U^P(t,x)=\bar{v}(x).
\]
Since $\Psi^P\leq u_0< \bar{v}$, we have $U^P(t,x)\leq u(t,x)<\bar{v}(x)$ by the comparison principle. Hence~\eqref{limit-u} holds. The theorem is proved.
\end{proof}


\section{Proof for blocking}\label{s:blocking}

In this section we prove Theorems~\ref{th:blocking-localized} and \ref{th:blocking-periodic} on the blocking of fronts. We first consider the case where the holes of $K$ are localized (Theorem~\ref{th:blocking-localized}), then dicuss the case where $K$ is periodic in $y$ (Theorem~\ref{th:blocking-periodic}). The two cases can be treated almost in parallel, with only a minor modification.

The blocking phenomenon in bistable equations caused by a narrow passage was first demonstrated rigorously by Matano \cite{M} for {\em dumbbell-shaped} bounded domains. Later, similar blocking results in bounded dumbbell-shaped domains were obtained by many authors including \cite{MM, J}. In the mean while, the works of  Berestycki, Bouhours and Chapuisat \cite{BBC} and Hamel and Zhang~\cite{HZ} discuss blocking phenomena for traveling waves that propagate through a semi-infinite cylinder-shaped domain (for $x_1\le0$) having a wide opening at the end, with bounded or unbounded section as $x_1\to+\infty$.

As mentioned in Introduction, the main idea of the proof of blocking in the present paper is to construct a barrier function, denoted by $w_0$, that is a stationary super-solution of the elliptic equation in the region $\Omega_{-1}=\Omega\cap\{x_1>-1\}$ with $w_{0}(-1, y) = 1$ for all $y\in \R^{N-1}$ and $w_{0}(x_{1}, y) \to 0$ as $x_{1}\to+\infty$. More precisely, $w_0$ is a solution of the problem
\begin{equation}\label{B}
\left\{
\begin{array}{ll}
\Delta w_0 +f(w_0)=0, \ \ &x\in \Omega_{-1} := \Omega\cap \{ x_1 > -1 \},   \\[4pt]
\dfrac{\partial w_0}{\partial \nu} = 0, \ \ &x\in \partial K =  \partial\Omega_{-1}\cap\{x_1 > - 1\},  \\[8pt]
w_0(-1, y) = 1,  & y\in \R^{N-1},
\end{array}\right.
\end{equation}
satisfying
\begin{equation}\label{w-infty}
w_0(x_1,y)\to 0\quad \hbox{as}\ \ x_1\to+\infty.
\end{equation}
In the case where $K$ is ${\mathcal P}$\,-\,periodic, we further assume that
\begin{equation}\label{w-periodic}
w_0(x_1, y) \ \ \hbox{is ${\mathcal P}$\,-\,periodic in $y$}.
\end{equation}
Next we will compare $w_0$ and $\bar{u}$ in \eqref{ubar-phi} and, by using Lemma \ref{lem:comparison-ancient} (ii),  it will follow that $\bar{v}(x)\le w_0(x)$ for all $x\in\Omega_{-1}$. Therefore the existence of such a barrier function $w_0$ immediately implies blocking. 

We present two different methods for constructing the barrier function $w_0$. Both methods rely on a variational argument for the same energy functional but under different constraints. The first method uses a constraint that is an extension of the one found in \cite[Theorem~6.2]{M}. The second method goes along the same line as in \cite{BBC}, though the argument is given in a more precise and more general manner.  Each of the two approaches is interesting in its own right and they may lead to different ways of generalizing the results in the future studies. 


\subsection{Proof of blocking: the first approach}\label{ss:blocking1}

We begin with the case where the holes of $K$ are localized in a bounded area. The case where $K$ is periodic can be treated with only minor modifications.

To start with, we observe that we can extend $f$ linearly outside the interval $(0,1)$, that is, we set
\be\label{fextended}
\hbox{$f(s)= f'(0) s$ for $s<0$, and $f(s) = f'(1) (s-1)$ for $s > 1$.}
\ee
Define
\[
F(s):=\int_0^s f(r)dr.
\]
By the condition \eqref{f}, the function $-F$ possesses a local minimum at $s=0$, a global minimum at $s=1$, with
\[
0=-F(0)>-F(1),
\]
and one local maximum at $s=\alpha$, with $-F(\alpha)>0$. One has that
\[
-F(s)\to+\infty \quad\hbox{as}\ \ s\to\pm\infty.
\]
We choose constants $\mu>0$, $\sigma>0$ and $\delta\in(0,\alpha]$ such that
\begin{equation}\label{W2}
-F(s)+\mu (s-\delta)^2\geq \sigma  \quad \hbox{for any}\ \ s\in\R.
\end{equation}
Such constants certainly exist since $-F(s)> 0$ for $s\in(0,\alpha]$ and $-F$ is bounded from below. From \eqref{W2} and $-F\ge0$ in $(-\infty,\alpha]$, it follows that
\begin{equation}\label{W1}
-F(s)+\mu (s-\bar s)^2\geq 0 \quad \hbox{for any}\ \ s\in\R, \ \bar s\in(-\infty,\delta].
\end{equation}

Next we decompose $\Omega_b :=\Omega\cap\{x_1> b\}$ into a union of bounded subdomains $D_j\subset\Omega_b \;(j=1,2,3,\ldots)$ with uniformly Lipschitz boundaries such that
\begin{equation}\label{Dj-Omega}
D_i\cap D_j=\emptyset\ (i\ne j),\quad \ \Omega_b \subset \bigcup_{j=1}^\infty \overline{D}_j,
\end{equation}
\begin{equation}\label{Dj-PW}
\int_{D_j}|\nabla \phi|^2 dx \geq 2\mu \int_{D_j}\left(\phi-\bar\phi\right)^2 dx,\quad 
{}^\forall\phi\in H^1(D_j) \ \ (j=1,2,3,\ldots),
\end{equation}
\begin{equation}\label{Dj-min}
D_{\min} :=\inf_{j}|D_j|>0,
\end{equation}
where $\bar\phi$ denotes the average of $\phi$ over $D_j$. The inequality \eqref{Dj-PW} is the so-called Poincar\'e--Wirtinger inequality, thus the condition \eqref{Dj-PW} implies that the first positive eigenvalue of $-\Delta$ on $D_j$ under the Neumann boundary conditions, denoted by $\mu_1(D_j)$, satisfies
\begin{equation}\label{mu1(Dj)}\tag{\ref{Dj-PW}'}
\mu_1(D_j)\geq 2\mu\quad \hbox{for all}\ \ j\in\N.
\end{equation}

\begin{Rk}[About the condition \eqref{Dj-PW}]\label{rk:Dj}
By the Szeg\H{o}--Weinberger inequality \cite{W}, we have $\mu_1(D_j)\leq \mu_1(B)$, where $B$ is a ball of the same volume as $D_j$. Thus
\[
\mu_1(D_j)\leq \left(\frac{|B_1|}{|D_j|}\right)^{\frac{2}{N}} \mu_1(B_1),
\]
where $B_1$ denotes the unit ball in $\R^N$. This and \eqref{mu1(Dj)}, together with \eqref{Dj-min}, imply
\begin{equation}\label{Dj-volume}
D_{\min}\leq |D_j|\leq \left(\frac{\mu_1(B_1)}{2\mu}\right)^{\frac N2} |B_1|.
\end{equation}
Thuerefore the volume of $D_j$ is uniformly bounded from below and above. This volume constraint is not enough to guarantee \eqref{Dj-PW} or \eqref{mu1(Dj)}. For example, even if $D_j$ has a proper volume, $\mu_1(D_j)$ can be very small if it is a long thin object, or if it is dumbell-shaped with a narrow middle part. On the other hand, if $D_j$ is a convex domain, then, by \cite{PW}, we have
\[
\mu_1(D_j)\geq \pi^2 \left({\rm diam}(D_j)\right)^{-2},
\]
where ${\rm diam}(D_j)$ denotes the maximal diameter of $D_j$, hence \eqref{mu1(Dj)} may be satisfied if ${\rm diam}(D_j)$ is not too large. Summarizing, the conditions \eqref{Dj-PW} (or \eqref{mu1(Dj)}) and \eqref{Dj-min} are fulfilled if the volume of $D_j$ lies in a certain range, if the maximal diameter of $D_j$ is not too large, and if it has a relatively regular shape such as convexity. Since we are assuming that $\Omega_b$ is a uniformly Lipschitz domain, such a decomposition exists.  What we have to pay attention is only the region near the wall $K$. In the region away from $K$, we can simply set~$D_j$ to be rectangles or parallelepiped domains of the same size.    
\end{Rk}

Now we are ready to prove the main results of this section.
\vskip 8pt
\noindent
\underbar{\bf The case of localized holes}:

\vskip 6pt 
\begin{proof}[Proof of Theorem~\ref{th:blocking-localized}]
We first recall the assumption of the theorem: $\Omega_{a,b}$ is bounded and
\begin{equation}\label{ab-epsilon}
|\Omega_{a,b}| \leq \ep,
\end{equation}
where the value of $\ep$ is to be specified later. We define a functional 
\begin{equation}\label{J-1}
\begin{split}
J_{-1}(w):=\int_{\Omega\cap\{-1<x_1<b\}}\left( \frac{|\nabla w|^2}{2}-F(w)+F(1)\right)dx+\int_{\Omega_b }\left( \frac{|\nabla w|^2}{2}-F(w)\right)dx,
\end{split}
\end{equation}
where $w$ varies in the set
\begin{equation*}
X:=\{w\in H^1_{loc}(\overline{\Omega_{-1}})\mid w(-1,y)=1 \ \hbox{for (almost) all}\ y\in\R^{N-1}\}.
\end{equation*}

If $w$ satisfies
\begin{equation}\label{Dj-delta}
\bar{w}_j:=\frac{1}{|D_j|} \int_{D_j} w\hspace{2pt} dx \leq \delta \quad(j=1,2,3,\ldots),
\end{equation}
then, by \eqref{W1} and \eqref{Dj-PW}, we have
\begin{equation}\label{Dj-J-w}
\int_{D_j}\left( \frac{|\nabla w|^2}{2}-F(w)\right)dx\geq  \int_{D_j}\left(\mu (w-\bar{w}_j)^2-F(w)\right)dx \geq 0.
\end{equation}
Therefore the following sum is well-defined
\[
0\leq \sum_{j=1}^\infty \int_{D_j}\left( \frac{|\nabla w|^2}{2}-F(w)\right)dx \leq+\infty.
\]
Thus the second term on the right-hand side of \eqref{J-1} can be defined as the following sum:
\[
\int_{\Omega_b }\left( \frac{|\nabla w|^2}{2}-F(w)\right)dx:=
\sum_{j=1}^\infty \int_{D_j}\left( \frac{|\nabla w|^2}{2}-F(w)\right)dx.
\]
Furthermore, since $W(s)$ attains its global minimum at $s=1$, we have
\begin{equation}\label{J-1-w}
\int_{\Omega\cap\{-1<x_1<b\}}\left( \frac{|\nabla w|^2}{2}-F(w)+F(1)\right)dx\geq 0.
\end{equation}
Therefore the functional $J_{-1}(w)$ in \eqref{J-1} is well-defined with values in $[0,+\infty]$, provided that $w$ satisfies the constraint \eqref{Dj-delta}. For each $\delta\in(0,\alpha]$, we define:
\begin{equation}\label{X-delta} 
X_\delta:=\{w\in X\mid \hbox{$w$ satisfies \eqref{Dj-delta}} \}.
\end{equation}

Next we define a function $\zeta$ by
\begin{equation}\label{zeta1}
\zeta (x) = \zeta (x_1,y)= 
\begin{cases}
\ 1 \quad &\text{if}\quad -1 \leq x_1 \leq a, \\[2pt]
\ \dfrac{b- x_1}{b-a} \ \  &\text{if}\quad a \leq x_1 \leq b, \\[2pt]
\ \, 0 \quad& \text{if}\quad x_1\geq b .
\end{cases}
\end{equation}
This function clearly belongs to $X_\delta$ for any $\delta\in(0,\alpha]$ and satisfies
\[
J_{-1}(\zeta)=\int_{\Omega\cap\{a\leq x_1\leq b\}}\left( \frac{|\nabla \zeta|^2}{2}-F(\zeta)+F(1)\right)dx
\leq \ep\left(\frac{1}{2(b-a)^2}-F(\alpha)+F(1)\right).
\]

Now we consider the following minimization problem:
\begin{equation}\label{min-J}
\underset{w\in X_\delta}{\rm Minimize}\  J_{-1}(w).
\end{equation}
The global minimizer of the above problem exists by virtue of \eqref{Dj-J-w} and \eqref{J-1-w}. Let $w_0$ be the global minimizer. Since $\zeta$ belongs to $X_\delta$, we have
\begin{equation}\label{J-w0}
0\leq J_{-1}(w_0) \leq J_{-1}(\zeta)\leq \ep\left(\frac{1}{2(b-a)^2}-F(\alpha)+F(1)\right).
\end{equation}
We claim that the condition \eqref{Dj-delta} holds strictly for $w_0$, namely
\begin{equation}\label{Dj-delta-0}
\overline{w_0}|_{D_j}:=\frac{1}{|D_j|} \int_{D_j} w_0\hspace{2pt} dx < \delta \quad(j=1,2,3,\ldots),
\end{equation}
if $\ep$ is chosen sufficiently small. To see this, suppose the contrary. Then 
\begin{equation}\label{Dj-w02}
\overline{w_0}|_{D_{j^*}}=\frac{1}{|D_{j^*}|} \int_{D_{j^*}} w_0 \hspace{2pt} dx = \delta \quad \hbox{for some} \ \ j^*\in\N.
\end{equation}
By \eqref{W2}, \eqref{Dj-PW}, and \eqref{Dj-w02}, we have
\[
\int_{D_{j^*}}\left( \frac{|\nabla w_0|^2}{2}-F(w_0)\right)dx\geq  \int_{D_{j^*}}\left(\mu (w_0-\overline{w_0}|_{D_{j^*}})^2-F(w_0)\right)dx \geq \sigma |D_{j^{*}}|,
\]
while, by \eqref{W1} and \eqref{Dj-delta}, we have, for all $j\in\N$,
\[
\int_{D_j}\left( \frac{|\nabla w_0|^2}{2}-F(w_0)\right)dx\geq  \int_{D_j}\left(\mu (w_0-\overline{w_0}|_{D_j})^2-F(w_0)\right)dx \geq 0.
\]
This, together with \eqref{J-1-w}, implies
\begin{equation}\label{J-w03}
J_{-1}(w_0)\geq \sigma |D_{j^{*}}|\geq \sigma D_{\min}.
\end{equation}

Now we choose $\ep>0$ sufficiently small so that
\begin{equation}\label{ep}
\ep< \sigma D_{\min} \left(\frac{1}{2(b-a)^2}-F(\alpha)+F(1)\right)^{-1}.
\end{equation}
Then \eqref{J-w03} contradicts \eqref{J-w0}. This contradiction proves that \eqref{Dj-delta-0} holds, provided that $\ep$ has been chosen to satisfy \eqref{ep}. This implies that $w_0$ lies in the interior of $X_\delta$ in the $H^1_{loc}$ topology, hence it satisfies the following Euler--Lagrange equation: 
\[
\left\{
\begin{array}{ll}
\Delta w_0 +f(w_0)=0, \ \ & x\in \Omega_{-1},\\[3pt]
\dfrac{\partial w_0}{\partial \nu} = 0, \ \ & x\in \partial\Omega_{-1}\cap\{x_1>-1\},\\[8pt]
w_0(-1,y)=1,\ \ &y\in\R^{N-1},
\end{array}\right.
\]
along with the constraint \eqref{Dj-delta-0}. 

Next we compare $w_0$ and $\bar{u}$ in \eqref{ubar-phi}. If we choose $T$ sufficiently negative, we have
\be\label{barudelta0}
0<\bar{u}(t,x)\leq \delta_0 \quad \hbox{for all}\ \ t\in (-\infty,T],\,x\in\Omega_{-1}=\Omega\cap \{x_1> -1\},
\ee
where $\delta_0$ is the constant introduced in \eqref{delta}. Furthermore,
\be\label{baruw0}
\bar{u}(t,x)<w_0(x)=1 \quad \hbox{for all}\ \ t\in(-\infty,T],\,x\in\Omega\cap \{x_1=-1\}.
\ee
Thus, by Lemma \ref{lem:comparison-ancient} (ii), we have $\bar{u}(t,x)<w_0(x)$ for all $x\in\Omega_{-1}$ and $t\in(-\infty,T]$, hence for all $t\in\R$ by the comparison principle. It follows that $\bar{v}(x)\le w_0(x)$ for all $x\in\Omega_{-1}$ . This, together with \eqref{Dj-delta-0}, implies that the propagation does not occur. In other words, we do not have $\bar{v}(x)\to 1$ as $x_1\to+\infty$.  Together with Theorem~\ref{th:dichotomy}, the proof of Theorem~\ref{th:blocking-localized} is complete.
\end{proof}

\begin{Rk}\label{rk:V-0}
Though we did not need it in the above proof, we can further show that $w_0(x)\to 0$ as $x_1\to+\infty$. Indeed, by \eqref{Dj-J-w}, we have
\[
+\infty>J_{-1}(w_0)\geq \sum_{j=1}^\infty \int_{D_j}\left( \frac{|\nabla w_0|^2}{2}-F(w_0)\right)dx\geq \sum_{j=1}^\infty\int_{D_j}\left(\mu (w_0-\overline{w_0}|_{D_j})^2-F(w_0)\right)dx,
\]
where $\overline{w_0}|_{D_j}$ denotes the average of $w_0$ on $D_j$. Hence
\[
\lim_{j\to+\infty} \int_{D_j}\left(\mu (w_0-\overline{w_0}|_{D_j})^2-F(w_0)\right)dx=0.
\] 
One can then deduce that $\overline{w_0}|_{D_j}\to 0$ as $j\to+\infty$ since $\mu(s-\bar s)^2-F(s)>0$ for $s\in\R$ if $0<\bar s\leq \delta$. Elliptic estimates then imply $w_0(x)\to 0$ as $x_1\to+\infty$.
\end{Rk}

\vskip 8pt
\noindent
\underbar{\bf The case where $K$ is periodic}:
\vskip 6pt

The proof of Theorem~\ref{th:blocking-periodic} goes completely in parallel to the proof of Theorem~\ref{th:blocking-localized}. The only difference is that the decomposition of the domain $\Omega_b$ given in \eqref{Dj-Omega} is now done in the unit cylinder of periodicity $\Delta_{\mathcal P}=\R\times{\mathcal C}_{\mathcal P}$, where ${\mathcal C}_{\mathcal P}$ is as in \eqref{unit-cell}, and the energy $J_{-1}$ is also defined in $\Delta_{\mathcal P}$. Thus each $D_j$ is a subset of $\Omega_b\cap\Delta_{\mathcal P}$ and the condition \eqref{Dj-Omega} is replaced by 
\begin{equation}\label{Dj-Omega-P}
D_i\cap D_j=\emptyset\ (i\ne j),\quad \ \Omega_b\cap\Delta_{\mathcal P}  \subset \bigcup_{j=1}^\infty \overline{D}_j.\quad |\partial D_j| =0 \ \ (j=1,2,3,\ldots),
\end{equation}
and the energy functional $J_{-1}$ is given in the form
\[
\begin{split}
J_{-1}(w):=\int_{\Omega\cap\Delta_{\mathcal P}\cap\{-1<x_1<b\}}\left( \frac{|\nabla w|^2}{2}-F(w)+F(1)\right)dx+\int_{\Omega_b \cap\Delta_{\mathcal P}}\left( \frac{|\nabla w|^2}{2}-F(w)\right)dx.
\end{split}
\]
Apart from these obvious modifications, the proof of Theorem~\ref{th:blocking-periodic} can be carried out in completely the same manner as that of Theorem~\ref{th:blocking-localized}, so we omit the details.


\subsection{Proof of Theorem \ref{th:blocking-periodic}, the second approach}\label{ss:blocking2}

In this subsection we explain the second approach for the proof of blocking. We consider only the periodic case. As mentioned earlier, this approach is similar to the one found in \cite{BBC}. 

To construct the barrier $w_0$ satisfying \eqref{B} and \eqref{w-infty}, we first consider a solution $w= w_R$ of the following problem in the region 
$\Omega_{-1,R}= \Omega \cap \{-1 < x_1 < R\}$, for $R\geq M+1$, that vanishes on the right hand side boundary:
\begin{equation}\label{BR}
\left\{
\begin{array}{ll}
\Delta w +f(w)=0, \ \ &x\in \Omega_{-1, R},   \\[4pt]
\dfrac{\partial w}{\partial \nu} = 0, \ \ &x\in \partial K =  \partial\Omega_{-1, R}\cap\{-1 < x_1 < R\},  \\[6pt]
w(-1, y) = 1, &y\in \R^{N-1},\\ [4pt]
w(R, y) = 0, &y\in \R^{N-1},\\ [4pt]
w(x_1, \cdot) \quad\text{is}\quad {\mathcal P}\text{\,-\,periodic}.&\\
\end{array}\right.
\end{equation}
The idea is to construct a function that is close to 1 in $\Omega_{-1, a}$, close to 0 in $\Omega_{b, R}$ and has a transition from 1 to 0 in $\Omega_{a,b}$  that is not much costly in terms of energy when the trace of this set in one periodicity cell is small in measure.

By the maximum principle we know that all solutions $w$ of \eqref{BR} with the function $f$ extended as in~\eqref{fextended} satisfy $0 \leq w \leq 1$, whence are solutions with the original function $f$. We then require some notations. We introduce the restrictions of the sets $\Omega, \Omega_\alpha, \Omega_{a,b}$ to one periodicity cell. That is, we denote:
\[
D:= \Omega\ \cap\ \R\!\times\!{\mathcal C}_{\mathcal P} ; \quad D_\alpha := \{ (x_1, y) \in \Omega\, ; \;  x_1 > \alpha,\;  y \in {\mathcal C}_{\mathcal P}\} ;
\]   
\[ 
D_{\alpha, \beta} := \{ (x_1, y) \in \Omega\, ; \;  \alpha< x_1 < \beta, \; y \in {\mathcal C}_{\mathcal P}\}=\Omega_{\alpha, \beta}\ \cap\ (\alpha, \beta)\!\times\!{\mathcal C} _{\mathcal P}. 
\]
For simplicity, we denote $D^R= D_{-1, R}$. Let $F(z):= \int_0^z f(s)ds$ and consider the functional:
\[
J (w):=  \int_{D^{R}} \frac{ |\nabla w |^2}{2} - F(w).
\]

Consider the function $\zeta= \zeta (x_1)$ defined by:
\begin{equation}\label{zeta}
\zeta (x) = \zeta (x_1)= 
\begin{cases}
\ 1 \quad &\text{if}\quad -1 \leq x_1 \leq a ; \\
\ \dfrac{b- x_1}{b-a} \ \  &\text{if}\quad a \leq x_1 \leq b ; \\
\ \, 0 \quad& \text{if}\quad b \leq x_1 \leq R.
\end{cases}
\end{equation}

We are going to locate a {\em local} minimum of $J$ over the space of ${\mathcal P}$-periodic functions in~$H^1(D^R)$ that satisfy the limiting conditions at $x_1= 1$ and $x_1= R$, and which is close to~$\zeta$. We denote by ${\mathcal H}^1_{\mathcal P}$ the space of ${\mathcal P}$-periodic functions over $\Omega^R := \Omega_{-1, R}$ that are in $H^1(D^R)$ and let
\[
{\mathcal H}={\mathcal H}_R := \{ u \in {\mathcal H}^1_{\mathcal P} ; \;  u\in H^1(D^R) ,\;
u(-1, y) = 1, \, u(R,y) = 0, \;\text{for (almost) all} \; y\in\R^{N-1} \}.
\]

We consider the problem of finding a local minimum of $J$ over ${\mathcal H}_R$, at least when $\ep$ is sufficiently small. Here is the key result in this direction. 

\begin{Prop}\label{locmin}
Given the domain $\Omega_b$ and $b-a$, there exist $\delta >0$ and $\ep>0$ such that if $|D_{a,b}| \leq \ep$, then for all $R\geq M+1$ we have 
\[
\inf \; \{J(w) \, ; \;\, w\in{\mathcal H}_R, || w - \zeta ||_{H^1(D^R)} = \delta \} > J(\zeta).
\]
\end{Prop}

\begin{proof}
It should be emphasized that $\delta$ and $\ep$ are independent of $R\geq M+1$. We denote $w= \zeta + v$ and thus $v$ is in the space ${\mathcal H}_{0,R}$ corresponding to the limiting conditions $v(-1,y) = v(R,y) = 0$ for (almost) all $y$, with periodicity in $y$. We split the functional $J$ in three parts $J^1, J^2, J^3$ corresponding to integration over the  the domains $D_{-1,a}, D_{a,b}$ and $D_{b,R}$ respectively. 

\vskip 0.2cm
\noindent{\bf 1. Estimate for $J^1$}.
From the assumptions on $f$, it follows that there exists some $\gamma >0$ such that
\begin{equation*}
- F(z) \geq - F(1) +  \gamma (z-1)^2 \; \text{for all} \; z \in \R.
\end{equation*}
Therefore, reducing $\gamma$ if need be, we can assume that $\gamma < 1/2$, and we get
\begin{equation}\label{J1}
J^1(w)= \int_{D_{-1, a}} \frac{ |\nabla w |^2}{2} - F(w)= J^1( 1+ v) \geq \gamma ||v||_{H^1(D_{-1, a})}^2 + J^1(\zeta).
\end{equation}

\vskip 0.2cm
\noindent{\bf 2. Estimate for $J^2$}.
We first note that
\[
 \int_{D_{a, b}} \frac{ |\nabla w |^2}{2} \geq \int_{D_{a, b}} \frac{ |\nabla v |^2}{4}- \int_{D_{a, b}} \frac{ |\nabla \zeta |^2}{2} =\int_{D_{a, b}} \frac{ |\nabla v |^2}{4}- \frac{1}{2(b-a)^{2} }|D_{a, b}|.
\]
Therefore, if $|D_{a,b}| \leq \ep$ we get
\[
\int_{D_{a, b}} \frac{ |\nabla w |^2}{2} \geq \int_{D_{a, b}} \frac{ |\nabla v |^2}{4}+ \int_{D_{a, b}} \frac{ |\nabla \zeta |^2}{2}- C \ep,
\] 
where $C>0$ is a generic constant. Using the same inequality as above for $-F$, we get
\[
\int_{D_{a,b}}\!\!\!\!\!-F(w) \geq \int_{D_{a,b}}\!\!\!\Big(- F(\zeta) + \gamma \frac{v^2}{2} - \gamma (\zeta -1)^2 - F(1) + F(\zeta)\Big)\geq \int_{D_{a,b}}\!\!\!\Big(- F(\zeta) + \gamma \frac{v^2}{2}\Big)\ - C \ep.
\]
Combining the two preceding inequalities yields:
\begin{equation}\label{J2}
J^2(w)= \int_{D_{a, b}} \frac{ |\nabla w |^2}{2} - F(w)\geq \frac{\gamma}{2} ||v||_{H^1(D_{a, b})}^2 + J^2(\zeta) -C \ep.
\end{equation}

\vskip 0.2cm
\noindent{\bf 3. Estimate for $J^3$}. Owing to the assumption $f'(0) <0$, $0$ is a strict local minimum for $-F$. Therefore, for $q>2$ given, reducing $\gamma$ if need be, there exists $C>0$ such that
\[
-F(z) \geq \gamma z^2 - C |z|^q, \; \text{for all} \; z\in\R.
\]
We choose $q$ to be defined by $q := \frac{2N}{N-2}$ when $N\geq 3$ and $q=4$ when $N=2$. Since $\zeta(x)=0$ in $D_{b, R}$ and choosing $\gamma < 1/2$, we see that
\[
J^3(w) = \int_{D_{b, R}} \frac{ |\nabla w |^2}{2} - F(w) \geq    \gamma||v||_{H^1(D_{b, R})}^2- C \int_{D_{b, R}}   |v|^q.
\]
Next, we require the following consequence of the Gagliardo-Nirenberg and Sobolev inequalities.

\begin{Lemma}
There is a constant $C_q>0$, independent of $R\geq M+1$, such that for any function $u$ in $H^1(D_{b, R})$, periodic in $y$, that vanishes on $y=R$ there holds
\[
\|u\|_{L^q(D_{b, R})}\leq C_q\|u\|_{H^1(D_{b, R})}.
\]
\end{Lemma}

We wish to emphasize that in these inequalities, the constant $C$ depends on the domain~$D_b$, in particular its Lipschitz norm, but it does not depend on $R$. These inequalities follow from the classical Gagliardo-Nirenberg and Sobolev inequalities (and   known as Ladyzhenskaya's inequality in the particular case $N=2, q=4$).  

Combining the previous inequalities yields:
\begin{equation}\label{J3}
J^3(w)= \int_{D_{b, R}} \frac{ |\nabla w |^2}{2} - F(w)\geq \gamma ||v||_{H^1(D_{b, R})}^2 - C ||w||_{H^1(D_{b, R})}^q +J^3(\zeta).
\end{equation}

\vskip 0.2cm
\noindent{\bf 4. Conclusion of the proof of Proposition~\ref{locmin}}. Combining the three estimates for $J^1, J^2$ and $J^3$, we obtain:

\begin{equation}\label{Jmin}
J(w)= \int_{D^R} \frac{ |\nabla w |^2}{2} - F(w)\geq \frac{\gamma}{2} ||w-\zeta||_{H^1(D^R)}^2  - C ||w-\zeta||_{H^1(D^R)}^q  -C \ep + J(\zeta).
\end{equation}
By choosing adequately $\delta>0$ we can make $\frac{\gamma}{2}\delta^2 - C \delta^q >0$. Then for $\ep>0$ small enough, this inequality proves Proposition~\ref{locmin}.
\end{proof}

\begin{proof}[Conclusion of the proof of Theorem \ref{th:blocking-periodic}]

We choose $\delta>0$ and $\ep >0$ as in the previous proposition. Note that they are independent of $R\geq M+1$. For any such $R$ we consider the minimization problem
\begin{equation}\label{minim}
\min \{ J(w) \; ; w \in {\mathcal H}_R, \; \|w- \zeta\|_{H^1(D^R)}\leq \delta\}.
\end{equation}
By standard arguments, we know that $J$ achieves its minimum on the closed ball of radius $\delta$ centered on $\zeta$ in ${\mathcal H}_R$. From Proposition~\ref{locmin}, it follows that
\[
J(\zeta) < \min \{ J(w) \; ; w \in {\mathcal H}_R, \; ||w- \zeta||_{H^1(D^R)}= \delta\}.
\]
This implies that the minimum of $J$ in the ball of radius $\delta$ about $\zeta$ is necessarily an {\em interior} minimum. Hence it is a local minimum of the energy $J$. 

For all $R\geq M+1$, we have thus found a solution $w_R$ of problem~\eqref{BR} that furthermore satisfies
\[
||w_R - \zeta||_{H^1(D^R)}\leq \delta.
\]
Since $\delta >0$ is independent of  $R$, we see that $w_R$ is bounded in $H^1$ norm, independently of $R$. Therefore, we can extract a subsequence $w_{R_j}$ of $w_R$ that converges weakly to some function~$w_0$ in $H^1(D_{-1})$. This function is a solution of problem~\eqref{B}. Moreover, as a limit of local minima, it is also a {\em stable} solution of problem~\eqref{B}. Arguing precisely as in the proof of Theorem~\ref{th:dichotomy} in Subsection~\ref{ss:proof-dichotomy}, we see that either $w_0(x_1, y) \to 0$ or $w_0(x_1, y) \to 1$ as $x_1\to+\infty$. But we know that the latter is impossible since $w_0$ is in~$L^2(D_{-1})$, whence $w_0(x_1, y) \to 0$ as $x_1\to+\infty$. The construction of the barrier function~$w_0$ is thus complete. 

In order to conclude the proof of Theorem~\ref{th:blocking-periodic}, it suffices to show that $\bar{v}\le w_0$ in the region~$\Omega_{-1}$. This follows from the existence of $T$ sufficiently negative such that~\eqref{barudelta0}-\eqref{baruw0} hold, together with Lemma~\ref{lem:comparison-ancient} (ii). As this part of the argument is precisely the same as in the proof of Theorem~\ref{th:blocking-localized} in the previous subsection, we omit the details. The proof of Theorem~\ref{th:blocking-periodic} in the second approach is complete. 
\end{proof}


\section{Proofs for propagation}\label{s:propagation}

In this section, we prove Theorems \ref{th:large-holes}, \ref{th:small-capacity} and \ref{th:parallel} on the propagation of fronts for the following three types of walls:
\begin{itemize}
\item walls with large holes (Theorem~\ref{th:large-holes});
\item small-capacity walls (Theorem~\ref{th:small-capacity});
\item parallel-blade walls (Theorem~\ref{th:parallel}).
\end{itemize}


\subsection{Walls with large holes: proof of Theorem~\ref{th:large-holes}}\label{ss:proof-large-holes}

\begin{proof}[Proof of Theorem \ref{th:large-holes}]
Let $R_0$ be the constant defined in \eqref{R0} and let $\Psi$ be the solution of~\eqref{Psi0} for $R=R_0$. For each point $P\in\Omega$ with $d(P,K)\geq R_0$ (here $d(P,K)$ denotes the distance between $P$ and the set $K$), we define a function $\Psi^P$ on $\overline{\Omega}$ as follows: 
\begin{equation}\label{Psi-P}
\Psi^P(x)=
\begin{cases}
\, \Psi(x-P)\ \ & \hbox{if}\ \ |x-P|\leq R_0,\\
\, 0\ \ & \hbox{otherwise}.
\end{cases}
\end{equation}
Then $\Psi^P$ is continuous on $\overline{\Omega}$ and satisfies 
\[
0\leq \Psi^P(x)<1\ \ (x\in\overline{\Omega}),\quad \ \ \max \Psi^P=\Psi(0)>\alpha,
\]
\[
\Delta\Psi^P+f(\Psi^P)=0\quad (|x-P|<R_0).
\]
Clearly, for any $P\in\Omega$ with $d(P,K)\leq R_0$, $\Psi^P$ is a subsolution of \eqref{E}. 

Since $\bar{v}\to 1$ as $x_1\to-\infty$, and since $\max \Psi^P=\Psi(0)<1$, we have
\[
\Psi^P(x)<\bar{v}(x)
\]
if the $x_1$ coordinate of $P$ is sufficiently negative. Choose such $P$ and call it $P_0$. Next let $P(s)=(1-s)P_0+sP_1$ be the interpolation point between $P_0$ and $P_1$. As $s$ varies from $0$ to $1$ continuously, the graph of $\Psi^{P(s)}$ slides along the line segment $P_0P_1$. By the strong maximum principle and the fact that $\bar{v}>0$ and that $\Psi^{P(s)}$ is a compactly supported subsolution, we see that $\Psi^{P(s)}$ remains below $\bar{v}$ all the way to $s=1$. Hence $P_1<\bar{v}$.

Next we move $P$ continuously along the curve $\gamma$, from $P_1$ to $P_2$. Then, since the distance between each point on $\gamma$ and $K$ is bounded from below by $R_0$, $\Psi^P$ is a compactly supported subsolution. Thus, again by the strong maximum principle, $\Psi^P$ remains below $\bar{v}$ all the way to $P=P_2$, hence $\Psi^{P2}<\bar{v}$. Finally, let $Q$ be an arbitrary point in the region $\{x_1\geq M+R_0\}$ and connect $P_2$ and $Q$ by a continuous curve $\gamma_2$ that is outside the $R_0$-neighborhood of $K$. Then, by the same argument as above we see that $\Psi^{Q}<\bar{v}$. Consequently we have
\[
\bar{v}(x)>\Psi(0)>\alpha \quad\hbox{for all}\ \ x\in \{x_1\geq M+R_0\}.
\]
By a simple comparison argument, we easily see that this implies $\bar{v}\to 1$ as $x_1\to+\infty$ since $f(s)>0$ for $\alpha<s<1$. The proof of Theorem~\ref{th:large-holes} is complete.
\end{proof}


\subsection{Small capacity walls: proof of Theorem~\ref{th:small-capacity}}\label{ss:proof-small-capacity}

\begin{proof}[Proof of Theorem \ref{th:small-capacity}]
Suppose the contrary. Then there exist positive numbers $\ep_j\to 0$ such that blocking occurs for $K^{\ep_j}\,(j=1,2,\ldots)$. Let $v_j$ denote the limit profile for $K^{\ep_j}$ and let $v_j\to v_\infty$ (after taking a subsequence). Then $v_\infty$ satisfies 
\[
\Delta v_\infty+f(v_\infty)=0 \ \ \hbox{in}\ \ \R^N\setminus ({\mathcal K}_1\cup {\mathcal K}_0).
\]
Since ${\mathcal K}_0$ has capacity $0$ and $v_\infty$ is bounded, ${\mathcal K}_0$ is a removable singularity. Therefore
\[
\Delta v_\infty+f(v_\infty)=0 \ \ \hbox{in}\ \ \R^N\setminus{\mathcal K}_1.
\]
By the assumption, ${\mathcal K}_1$ has a passage of width larger than or equal to $R_0$. Therefore, as in the proof of Theorem~\ref{th:large-holes}, we see that $v_\infty \to 1$ as $x_1\to+\infty$. This, however, contradicts the statement of Corollary~\ref{cor:blocking}. This contradiction proves the theorem.
\end{proof}


\subsection{Parallel-blade walls: proof of Theorem~\ref{th:parallel}}\label{ss:proof-parallel}

The proof of propagation for this type of wall is based on a rather non-standard comparison argument using ``quasi-subsolutions'', the meaning of which will be explained below. In what follows, $\bar{v}^\ep$ will denote the limit profile corresponding to the wall $K^\ep$. We want to show that $\bar{v}^\ep(x)\to 1$ as $x_1\to+\infty$ if $\ep$ is sufficiently small.

We introduce a function $\rho=\rho(x_1)$ on $\R$ that 
will serve as the basis of our argument. Let $\delta>0$ be a small constant such that $s\mapsto f(s)-\delta$ satisfies
\[
\int_0^b \left(f(s)-\delta\right)ds >0,
\]
where $b$ is the stable zero of the function $s\mapsto f(s)-\delta$ such that $\alpha<b<1$. We define a function $\rho=\rho(x_1)$ on $\R$ by
\begin{equation}\label{rho}
\left\{\begin{array}{ll}
\rho''+f(\rho)= \delta\ \ &\hbox{if}\ \ x_1< 0\\[4pt]
\rho>0,\ \ \rho'<0 & \hbox{if}\ \ x_1< 0\\[4pt]
\rho(x_1)=0 \quad & \hbox{if}\ \ x_1\geq 0\\[4pt]
\rho(x_1)\to b & \hbox{as}\ \ x_1\to -\infty.
\end{array}\right.
\end{equation}
Such a function $\rho$ exists and is unique as long as $\delta>0$ is sufficiently small. This is easily seen by a shooting argument for the ordinary differential equation $u''+f(u)=\delta$, and we omit the details.

Next, for each $\lambda\in\R$, we define a function $\rho^\lambda$ on $\R^N$ by 
\[
\rho^\lambda(x):=\rho(x_1-\lambda).
\]
The support of $\rho^\lambda$ is the half space $\{x_1\leq \lambda\}$.  Since $\bar v(x)\to 1$ as $x_1\to-\infty$ and  $\sup \rho^\lambda=b<1$, it follows that for all sufficiently large $\lambda<0$ the function $\rho^\lambda$ is below $\bar{v}^\ep$, i.e.
\begin{equation}\label{rho<v}
\rho^\lambda(x)\leq \bar{v}^\ep(x) \ \ (x\in\Omega).
\end{equation}
We would like to show that this inequality continues to hold for all $\lambda$ by a sliding type argument. However, this is not necessarily true in general. Indeed, the function $\rho^\lambda$ is not exactly a subsolution of \eqref{E} for large $\lambda>0$ because it does not satisfy the appropriate boundary conditions $\partial_\nu\rho^\lambda\leq 0$ on some part of the boundary of $K^\ep$. Consequently, the inequality \eqref{rho<v} may not hold for all $x\in\Omega$ if $\lambda>0$. Instead, we are going to show that the region where $\rho^\lambda(x)$ is larger than $\bar{v}^\ep(x)$ remains small. More precisely, define
\[
D^\lambda:=\{x\in\overline{\Omega}\mid \rho^\lambda(x)>\bar{v}^\ep(x)\}\cap \Delta_{\mathcal P}
\]
We shall now prove that the volume of $D^\lambda$ remains small for all $\lambda>0$ provided $\ep$ is sufficiently small, which implies that \eqref{rho<v} ``nearly'' holds. 

Let $w^\lambda:=\rho^\lambda-\bar{v}^\ep$. Then $w^\lambda$ is of class $C^2$ on $\Omega$ and satisfies
\begin{equation}\label{w-lambda}
\left\{
\begin{array}{ll}
\Delta w^\lambda +f'(\xi(x))w^\lambda= \delta,  \ \ &\hbox{in}\ \ \Omega\\[4pt]
w^\lambda>0 \ \ &\hbox{in}\ \ D^\lambda,\\[4pt]
\displaystyle
w^\lambda=0, \ \ &\hbox{on}\ \ \partial D^\lambda\cap\Omega,\\[7pt]
\displaystyle
\frac{\partial w^\lambda}{\partial\nu}=\frac{\partial \rho^\lambda}{\partial\nu}\ \ \ &\hbox{on}\ \ \partial D^\lambda\cap \partial K^\ep.
\end{array}\right.
\end{equation}
As for the boundary conditions on $\partial\Delta_{\mathcal P}$ it is understood that we are dealing with ${\mathcal P}$-periodic functions and we will see that these conditions do not affect the computations that follow. 

\begin{Lemma}\label{wlambda=0}
The set $S(\lambda):=\{x\in\Omega\mid w^\lambda(x)=0\}\cap \Delta_{\mathcal P}$ has Lebesgue measure zero.
\end{Lemma}

\begin{proof}
Suppose that $S(\lambda)$ has a positive Lebesgue measure. Then by the Lebesgue density theorem, almost every point $x^*$ of $S(\lambda)$ is a density point, in the sense that
\begin{equation}\label{Lebesgue}
\lim_{r\to 0}\frac{|S(\lambda)\cap B_r(x^*)|}{|B_r(x^*)|}=1,
\end{equation}
where $B_r(x^*)$ denotes a ball of radius $r$ centered at $x^*$. This implies $\nabla w^\lambda(x^*)=0$, since otherwise $S(\lambda)$ would be a smooth hypersurface around $x^*$, thus \eqref{Lebesgue} would not hold. 

Next, since $\Delta w^\lambda=\delta>0$ at $x^*$, there exists a unit vector ${\boldsymbol e}$ such that
\[
\frac{d^2}{ds^2} w^\lambda(x^*+s{\boldsymbol e})>0.
\]
This, together with $w^\lambda(x^*)=0$ and $\nabla w^\lambda(x^*)=0$, implies that $w^\lambda(x)>0$ in the intersection of a small neighborhood of $x^*$ and the interior of a dual cone with vertex at $x^*$. This again contradicts \eqref{Lebesgue}. Therefore \eqref{Lebesgue} never holds, hence $S(\lambda)$ has Lebesgue measure zero.
\end{proof}

\begin{Cor}\label{cor:D-lambda-conti}
$|D^\lambda|$ is increasing and continuous in $\lambda$. 
\end{Cor}

\begin{proof}
Since $w^\lambda(x)=\rho^\lambda(x)-\bar{v}^\ep(x)$ is increasing in $\lambda$ for each fixed $x\in\Omega$, the set $D^\lambda$ is increasing in $\lambda$, hence so is $|D^\lambda|$. To prove continuity, observe that $w^\lambda(x)$ is continuous and increasing in $\lambda$ for each fixed $x$. Hence, for each $\lambda_0>0$,
\[
D^{\lambda_0}=\lim_{\lambda\to\lambda_0-0}D^\lambda=\bigcup_{0<\lambda<\lambda_0} D^\lambda.
\]
Consequently, we have
\[
|D^{\lambda_0}|=\lim_{\lambda\to\lambda_0-0}|D^\lambda|.
\]
Next, for each $\lambda_0\geq 0$, the continuity and monotonicity of $w^\lambda$ implies
\[
\lim_{\lambda\to\lambda_0+0}|D^\lambda|=|D^{\lambda_0}\cup S(\lambda_0)|.
\]
Since $S(\lambda_0)$ is a set of measure zero by Lemma~\ref{wlambda=0}, we have
\[
\lim_{\lambda\to\lambda_0+0}|D^\lambda|=|D^{\lambda_0}|.
\]
This establishes the continuity of $\lambda\mapsto |D^\lambda|$.
\end{proof}

Now we are ready to prove the following estimate on the volume of $D^\lambda$.

\begin{Prop}\label{prop:estimateD}
For any given constant $\nu>0$, there exists $\ep_0= \ep_0(\nu)$, also depending on $f$, such that if assumptions \eqref{P1}--\eqref{P2} are satisfied for any $\ep \leq \ep_0$, we have
\begin{equation}\label{estimateD}
|D^\lambda| \, \leq \, \nu, \quad\text{for all} \quad \lambda\in\R.
\end{equation}
\end{Prop}

\begin{proof}
Let $\lambda_0<0$ be such that $\bar{v}^\ep(x_1,y)\geq b$ for all $x_1\leq \lambda_0$ and all $y$. For $\lambda>\lambda_0$, define
\[ 
R^\lambda =\Omega \cap \{ (x_1, y) \in \Delta_{\mathcal P} , \; \lambda_0< x_1 < \lambda \}.
\]
The difference $w^\lambda$ satisfies a linear equation 
\begin{equation}\label{eqnw}
\Delta w^\lambda + q(x) w^\lambda = \delta \quad\text{in} \quad R^\lambda,
\end{equation}
for some function $q$  bounded by the Lipschitz norm of $f$:
\[
| q(x)| \leq L_f:= || f' ||_{L^\infty(0,1)}, \quad\text{for all}\quad x\in \Omega^\lambda.
\]

For each $\eta >0$, let $\chi_\eta$ be defined by:
\begin{equation*}
\chi_\eta(s) =
    \begin{cases}
    0,\quad \text{if} \quad s\leq 0,\\
    s/\eta,\quad \text{if} \quad 0\leq s\leq \eta,\\
    1,\quad \text{if} \quad s\geq \eta.\\
    \end{cases}
\end{equation*}
Multiply equation~\eqref{eqnw} by $\chi_\eta(w^\lambda)$ and integrate by parts on $R^\lambda$ to get:
\begin{equation}
\delta \int_{R^\lambda} \chi_\eta(w^\lambda) = - \int_{R^\lambda} \chi'(w^\lambda) |\nabla w^\lambda|^2 + \int_{\partial R^\lambda}\frac{\partial w^\lambda}{\partial\nu} \chi_\eta(w^\lambda)\,dS_x + \int_{R^\lambda} q(x) w^\lambda \chi_\eta(w^\lambda).  
\end{equation}
Observe that $w^\lambda \leq 0$ on $x_1= \lambda_0$ and $x_1 = \lambda$, as well as on $\partial R^\lambda\setminus K^\ep$, and that there is no contribution of the boundary of $\Delta_{\mathcal P}$ because the functions are periodic. Thus we get
\[
\int_{\partial R^\lambda}\frac{\partial w^\lambda}{\partial\nu} \chi_\eta(w^\lambda)\,dS_x =
\int_{\partial R^\lambda \cap K^\epsilon }\frac{\partial \rho^\lambda}{\partial\nu} \chi_\eta(w^\lambda)\,dS_x.
\]
Since $\partial \rho^\lambda/\partial \nu=\nabla \rho^\lambda\cdot \nu=\rho'(x_1-\lambda)\hspace{1pt}{\boldsymbol e}_1\cdot \nu$, the condition \eqref{P2} implies
\[
\int_{\partial R^\lambda\cap\partial K^\ep}\Big|\frac{\partial \rho^\lambda}{\partial\nu}\Big|\,dS_x\leq  \ep_1\|\rho'\|_{L^\infty}.
\]

Denoting $\|\rho'\|_{L^\infty}$ by $k$ this, together with $\chi_\eta(w^\lambda)=0$ when $w^\lambda\le0$, yields
\[
\delta \int_{R^\lambda} \chi_\eta(w^\lambda) \leq L_f \int_{D^\lambda} w^\lambda + k \,\ep_1.
\]
Then, letting $\eta\to 0$, we obtain
\begin{equation} \label{estD1}
\delta|D^\lambda|\leq L_f \int_{D^\lambda} w^\lambda + k \,\ep_1.
\end{equation}

To derive the second estimate we need, we multiply equation~\eqref{eqnw} by $w^+$ (to simplify notations, we now write $w$ instead of $w^\lambda$). Integration over $R^\lambda$ and Green's formula yield:
\[
\begin{split}
\delta \int_{D^\lambda} w \, & \, \leq  \int_{\partial R^\lambda}w^+ \, \frac{\partial w} {\partial\nu}\,dS_x
-\int_{D^\lambda}\left|\nabla w \right|^2 +L_f\int_{D^\lambda} w^2 \\
&=\int_{\partial R^\lambda\cap\partial K^\ep}w \frac{\partial \rho^\lambda}{\partial\nu}\,dS_x-\int_{D^\lambda}\left|\nabla w \right|^2 +L_f\int_{D^\lambda} w^2 
\end{split}
\]
where we have used the fact that $\int_{R^\lambda} \nabla w \cdot \nabla w^+\, =\, \int_{D^\lambda}\left|\nabla w \right|^2$. Again using condition \eqref{P2}, we obtain
\[
\delta \int_{D^\lambda}w \,\leq \,
-\int_{D^\lambda}\left|\nabla w\right|^2 +L_f\int_{D^\lambda} w^2 + k\,\ep_1.
\]
Combining this with \eqref{estD1} yields
\begin{equation}\label{D-lambda2}
|D^\lambda|\leq \delta^{-2}L_f\left(-\int_{D^\lambda}\left|\nabla w\right|^2 +L_f\int_{D^\lambda} w^2 \right)+ k \delta^{-1} (1 +L_f \delta^{-1}) \, \ep_1.
\end{equation}

Next we set $\widehat D^\lambda:=D^\lambda\setminus\left( [0,M]\times {\mathcal N}_{\ep}(\Sigma)\right)$. Then
\[
-\int_{D^\lambda}\left|\nabla w\right|^2 +L_f\int_{D^\lambda}w^2 \leq
-\int_{\widehat{D}^\lambda}\left|\nabla w\right|^2 +L_f\int_{\widehat D^\lambda}w^2 +C_\Sigma L_f \,\ep,
\]
where $C_\Sigma$ is a constant depending on $\Sigma$. This and \eqref{D-lambda2} yield
\[
|D^\lambda|\leq \delta^{-2} L_f\left(- \int_{\widehat D^\lambda}\left|\nabla w\right|^2 + L_f\int_{\widehat D^\lambda} w^2 \right)+C_1(\ep+ \ep_1),
\]
where the constant $C_1$ involves the various other parameters $\delta, L_f, k, C_\Sigma$ which are fixed and in particular do not depend on $\ep$.

By Lemma \ref{lem:poincare} below, there exists a constant $C>0$, independent of $\ep$, such that:
\[
|D^\lambda|\leq \delta^{-2}L_f \left(-C |\widehat{D}^\lambda|^{-2/N}+L_f\right)\int_{\widehat D^\lambda}w^2 +C_1(\ep+\ep_1)
\]
so long as $|\widehat{D}^\lambda|\leq \eta_0$. Choose $\eta_1>0$ sufficiently small such that $C \eta_1^{-2/N}\geq L_f$. Then we have
\[
|D^\lambda|\leq \eta_1 \ \Longrightarrow \ |D^\lambda|\leq C_1(\ep+\ep_1).
\]
Since, by the assumption in \eqref{P2}, $\ep_1=\ep_1(\ep)$ tends to $0$ as $\ep\to 0$, there exists $\ep^*$ such that  
\[
\eta_1>C_1(\ep+\ep_1) \ \ \ \hbox{for all} \ \ep\in(0,\ep^*].
\]
Thus, if $\ep\in(0,\ep^*]$, we have
\[
|D^\lambda|\leq \eta_1 \ \Longrightarrow \ |D^\lambda|\leq C_1(\ep+\ep_1) \ \big(\,<\eta_1\,),
\]
which implies that $|D^\lambda|$ cannot take a value between $C_1(\ep+\ep_1)$ and $\eta_1$. Now, as we increase the value of $\lambda$, the value of $|D^\lambda|$ for $\lambda\leq 0$ equals $0$ and it depends on $\lambda>0$ continuously by Corollary~\ref{cor:D-lambda-conti}. Consequently, we have
\[
|D^\lambda|\leq C_1(\ep+\ep_1) \ \ \hbox{for all}\ \ \lambda\in\R,
\]
provided that $0<\ep<\ep^*$. Since $C_1(\ep+\ep_1)$ can be arbitrarily small if $\ep$ is chosen small, the proof of Proposition~\ref{prop:estimateD} is complete.
\end{proof}

\begin{proof}[Completion of the proof of Theorem~\ref{th:parallel}]
The inequality \eqref{estimateD} implies that the limit profile $\bar v^\ep$ cannot tend to $0$ as $x_1\to+\infty$, since otherwise we would have $|D^\lambda|\to+\infty$ as $\lambda\to+\infty$. Hence propagation occurs for $K^\ep$ so long as $0<\ep<\ep^*$. The theorem is proved.
\end{proof}

Here is the lemma we used in the above proof. It will be proved in Section \ref{s:poincare} of Appendix.

\begin{Lemma}\label{lem:poincare}
There exist a constant $\eta_0>0$ that is independent of $\ep$ and a constant $C>0$, depending on $\eta_0$, but again independent of $\ep$, such that, for any nonnegative $w\in H^1_{\mathcal P}(\R^N\setminus \left( [0,M]\times {\mathcal N}_\ep(\Sigma)\right)$ the following inequality holds
\begin{equation}\label{poincare}
\int_{supp(w)\cap\Delta_{\mathcal P}}|\nabla w|^2 dx \geq C\, |supp(w)\cap\Delta_{\mathcal P}|^{-2/N} \int_{supp(w)\cap\Delta_{\mathcal P}}w^2 dx
\end{equation}
so long as $|supp(w)\cap\Delta_{\mathcal P}|\leq \eta_0$.
\end{Lemma}

\begin{Rk}\label{rk:quasi-subsolution}
There are indeed cases where the inequalty \eqref{rho<v} does not hold everywhere in $\Omega$ for large values of $\lambda$; in other words, we may have $D^\lambda\ne\emptyset$ for large $\lambda$. For example, if some parts of $K^\ep$ have tiny reservoir-shaped pockets with a sufficiently narrow entrance as shown in Figure~\ref{fig:pocket}, then the value of $\bar v$ remains small inside those pockets by Theorem~\ref{th:incomplete}, therefore the inequality \eqref{rho<v} cannot hold in those pockets. Note that the condition \eqref{P2} is satisfied provided the pockets are tiny enough.
\end{Rk}

\begin{figure}[h]
		\begin{center}
			\includegraphics[width=0.55\textwidth]
			{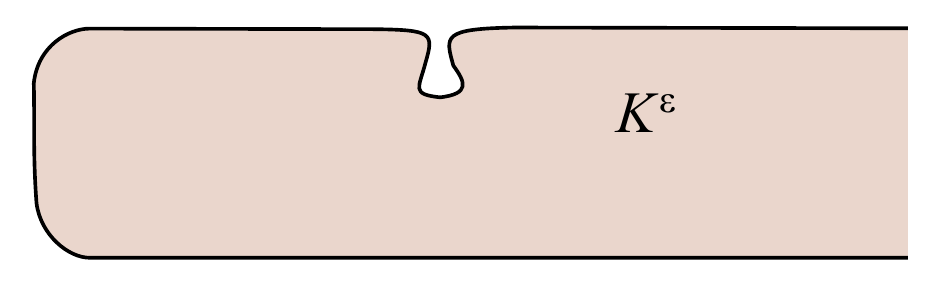}
		\end{center}
		\vspace{-10pt}
		\caption{A tiny pocket on the surface of $K^\ep$, which prevents complete invasion.}
		\label{fig:pocket}
\end{figure}


\section{Complete and incomplete invasions}\label{s:complete}

In this section we prove Theorem~\ref{th:complete} on a sufficient condition for complete invation and Theorem~\ref{th:incomplete} on an example of incomplete invasion.


\subsection{Proof of Theorem~\ref{th:complete} for complete invasion}\label{ss:proof-complete}

As we mentioned earlier, the proof is basically the same as that of Theorem 6.4 in \cite{BHM1}, though the notation here is simpler as we are considering directional convexity only in direction $x_1$. 

\begin{proof}[Proof of Theorem~\ref{th:complete}] 
Let $\rho=\rho(x_1)$ be as in \eqref{rho} and define a function $W^\lambda$ on $\R^N$ by
\begin{equation}\label{rho-lambda}
W^\lambda(x):=\max\{ \rho(x_1-\lambda), \, \rho(M-x_1-\lambda)\}.
\end{equation}
This function is symmetric with respect to the hyperplane $\{x_1=M/2\}$ and monotone decreasing (resp. increasing) in the region $\{x_1\leq M/2\}$ (resp. $\{x_1\geq M/2\}$). When $\lambda<0$, the support of~$W^\lambda$ is the set $\{x_1\leq\lambda\}\cup\{x_1\geq M-\lambda\}$, which does not touch the wall $K$. Therefore, $W^\lambda|_\Omega$ is a subsolution of \eqref{E} for all $\lambda<0$. Furthermore, since we are assuming that propagation occurs, we have $\bar{v}(x)\to 1$ as $x_1\to\pm\infty$. Therefore, we have
\begin{equation}\label{W>v}
W^\lambda|_\Omega(x)<\bar{v}(x)\quad \hbox{for all} \ \ x\in\Omega
\end{equation}
if $\lambda$ is sufficiently negative.

Now we increase $\lambda$ continuously. Since $W^\lambda|_\Omega$ is a subsolution for all $\lambda<0$, \eqref{W>v} continues to hold up to $\lambda=0$ by the strong maximum principle. 
Once $\lambda$ becomes positive, the support of $W^\lambda$ meets the wall $K$, but $W^\lambda|_\Omega$ remains to be a subsolution for all $\lambda> 0$. To see this, recall that $K$ is directionally convex and that $\rho$ is a monotone decreasing function. Therefore
\[
\frac{\partial W^\lambda}{\partial\nu}\leq 0 \quad \hbox{on}\ \ \partial\Omega\setminus\{x_1=M/2\}.
\]
Furthermore, at $x_1=M/2$,  $W^\lambda$ has a positive derivative gap. Consequently, $W^\lambda|_\Omega$ is a subsolution for all $\lambda>0$. We can therefore increase $\lambda$ continuously in the region $\lambda>0$, to obtain \eqref{W>v} for all $\lambda\in\R$. Hence
\[
\bar{v}(x) \geq \lim_{\lambda\to+\infty}W^\lambda(x)=b\quad \hbox{on}\ \ \Omega.
\]
Since $b>\alpha$, a simple comparison argument shows that $\bar{v}(x)=1$ everywhere. The proof of the theorem is complete.
\end{proof}


\subsection{Proof of Theorem~\ref{th:incomplete} for incomplete invasion}

Here we restate Theorem~\ref{th:incomplete} in a more precise manner and prove it.  The method of the proof is basically the same as the proof of Theorem~\ref{th:blocking-localized} for blocking given in Subsection~\ref{ss:blocking1}.  We first specify the structure of the ``reservoir''.

A typical image of a reservoir is shown in Figure~\ref{fig:incomplete} (left). The mouth of the reservoir can face in any direction. Let ${\boldsymbol e}$ denote the unit vector pointing toward the mouth of the reservoir. The reservoir consists of two open subdomains $\Omega_0$ and $\Omega_{-}$: the former represents the narrow entrance path and the latter the interior of the reservoir (Figure~\ref{fig:complete} (right)). Here $z:=x\cdot{\boldsymbol e}$ denotes the coordinate in the direction ${\boldsymbol e}$. We assume that
\[
\partial\Omega_0\cap\Omega\subset \{z=a\}\cup\{z=b\},\quad 
\partial\Omega_{-}\cap\Omega\subset \{z=b\},
\]
as shown in Figure~\ref{fig:incomplete} (right). We set $\Gamma:=\partial\Omega_0\cap\{z=a\}$, the outer-most boundary of $\Omega_0$.
The entire reservoir is the domain $\Omega_{res}:=\Omega_0\cup(\overline{\Omega}_{-}\cap\Omega)$.

\begin{figure}[h]
\begin{center}
\includegraphics[width=0.72\textwidth]{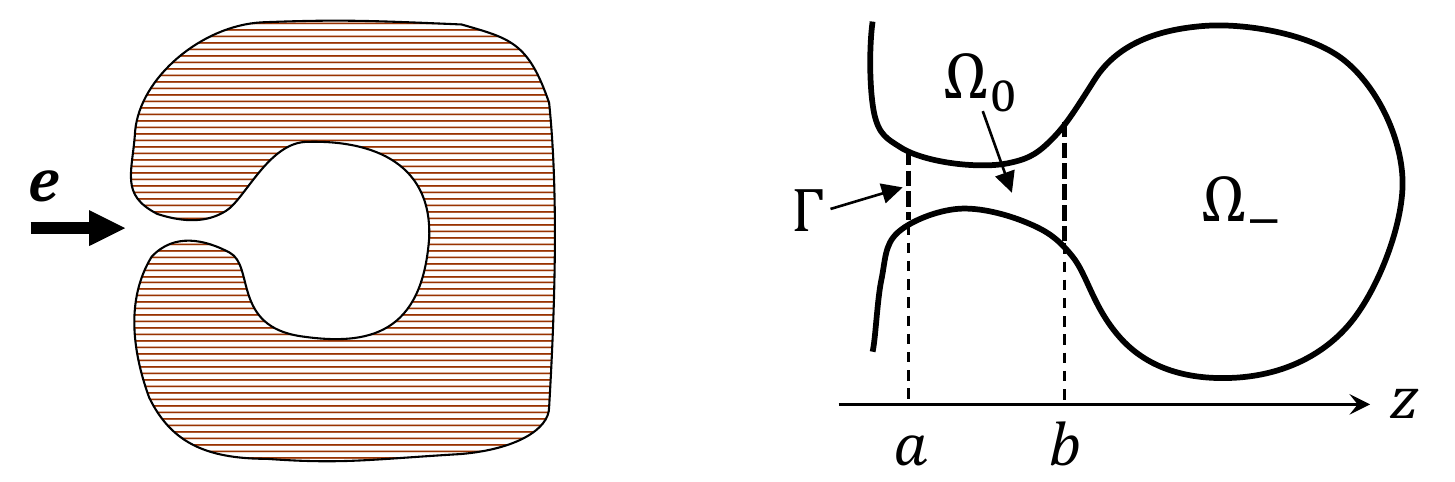}
\end{center}
\vspace{-10pt}
\caption{Example of a reservoir-like configuration and its magnified view.}
\label{fig:incomplete}
\end{figure}

Next let $\delta\in(0,\alpha]$, $\mu>0$, $\sigma>0$ be the constants satisfying \eqref{W2}-\eqref{W1}, that is,
\[
-F(s)+\mu (s-\delta)^2\geq \sigma  \quad \hbox{for any}\ \ s\in\R,
\]
\[
-F(s)+\mu (s-\bar s)^2\geq 0 \quad \hbox{for any}\ \ s\in\R, \ \bar s\in(-\infty,\delta],
\]
where $F(s)=\int_0^s f(r)dr$, with $f$ extended as in~\eqref{fextended}. We decompose the domain $\Omega_{-}$ into a union of subdomains $D_j\subset\Omega_{-} \;(j=1,2,\ldots, m)$ with uniformly Lipschitz boundaries such that
\[
D_i\cap D_j=\emptyset\ (i\ne j),\quad \ \Omega_{-} \subset \bigcup_{j=1}^m \overline{D}_j,
\]
\[
\int_{D_j}|\nabla \phi|^2 dx \geq 2\mu \int_{D_j}\left(\phi-\bar\phi\right)^2 dx\quad 
{}^\forall\phi\in H^1(D_j) \ \ (j=1,2,\ldots,m),
\]
where $\bar\phi$ denotes the average of $\phi$ over $D_j$. We set
\[
D_{\min}:=\min\{|D_2|,|D_2|.\ldots,|D_m|\}.
\]
See Remark~\ref{rk:Dj} on such decomposition and the meaning of the constant $\mu$. If the size of~$\Omega_{-}$ is very large, we need to split it into subdomains of appropriate sizes, otherwise we can simply set $m=1$ and $D_1=\Omega_{-}$.

Our goal is to construct an upper barrier that blocks the invasion of $\bar{u}$ into the reservoir. Such a barrier can be constructed as a stationary solution $V(x)$ of the following problem:
\begin{equation}\label{V2}
\left\{
\begin{array}{ll}
\Delta V +f(V)=0, \ \ & x\in \Omega_{res},\\[3pt]
V=1,\ \ & x\in\Gamma,\\[3pt]
\dfrac{\partial V}{\partial \nu} = 0, \ \ & x\in \partial\Omega_{res}\setminus\Gamma
\end{array}\right.
\end{equation}
that satisfies
\begin{equation}\label{Dj-V2}
\frac{1}{|D_j|} \int_{D_j} V(x) dx \leq \delta \quad \hbox{for all} \ \ \ j=1,2,\ldots,m.
\end{equation}
To do so, we first define the following functional which is an analog of $J_{-1}$ in \eqref{J-1}:
\[
J_{res}(w):=\int_{\Omega_0}\left( \frac{|\nabla w|^2}{2}-F(w)+F(1)\right)dx
+\int_{\Omega_{-}}\left( \frac{|\nabla w|^2}{2}-F(w)\right)dx
\]
and minimize this functional on the following set of functions
\[
X_\delta=\left\{ w\in H^1(\Omega_{res}) \,\Big|\, w_\Gamma=1,\; \frac{1}{|D_j|} \int_{D_j} w(x) dx \leq \delta \; (j=1,2,\ldots,m) \right\}.
\]
Let $V$ be the global minimizer of $J_{res}$ over the set $X_\delta$ and define
\[
\zeta (x) = 
\begin{cases}
\ \dfrac{b- z}{b-a} \ \  &\text{if}\quad x\in\Omega_0, \\[2pt]
\ \, 0 \quad& \text{if}\quad x\in\Omega_{-}.
\end{cases}
\]
Since $\zeta\in X_\delta$, we have
\[
J_{res}(V)\leq J_{res}(\zeta)\leq |\Omega_0|\left(\frac{1}{2(b-a)^2}-F(\alpha)+F(1)\right).
\]

Now assume that the following inequality holds:
\begin{equation}\label{ep2}
|\Omega_0|< \sigma D_{\min} \left(\frac{1}{2(b-a)^2}-F(\alpha)+F(1)\right)^{-1}.
\end{equation}
Then, arguing as in the proof of Theorem~\ref{th:blocking-localized}, we see that 
\eqref{Dj-V2} holds strictly, that is,
\[
\frac{1}{|D_j|} \int_{D_j} V(x) dx < \delta \quad \hbox{for all} \ \ \ j=1,2,\ldots,m.
\]
This implies that $V$ lies in the interior of $X_\delta$, hence it satisfies the Euler--Lagrange equation~\eqref{V2}.  Since $\Omega_{res}$ is bounded, $V$ is uniformly positive in $\Omega_{res}$. Therefore $\bar{u}(t,x)<V(x)$ in $\Omega_{res}$ for $t$ sufficiently negative. Hence, by the comparison principle, $\bar{u}(t,x)<V(x)$ for all $t\in\R$, which implies $\bar{v}(x)\leq V(x)$ in $\Omega_{res}$. Combining this and \eqref{Dj-V2}, we see that $\bar{v}$ is not identically equal to $1$. Summarizing, we have proved the following theorem.

\begin{Th}\label{th:incomplete2}
Let $\Omega_0$ and $\Omega_{-}$ be as above, and assume that \eqref{ep2} holds. Then $\bar{v}<1$. In particular, the complete invasion does not occur even if propagation takes place.
\end{Th}


\appendix

\section{Appendix: relative Poincar\'e inequality}
\label{s:poincare}

In this Appendix, we prove Lemma \ref{lem:poincare}, which we have used in the proof of Theorem \ref{th:parallel} in Subsection \ref{ss:proof-parallel}.  We restate this lemma:

\begin{Lemma}\label{lem:poincare-a}
There exist a constant $\eta_0>0$ that is independent of $\ep$ and a constant $C>0$, depending on $\eta_0$, but again independent of $\ep$, such that, for any nonnegative $w\in H^1_{\mathcal P}(\R^N\setminus \left( [0,M]\times {\mathcal N}_\ep(\Sigma)\right)$ the following inequality holds
\begin{equation}\label{poincare-a}
\int_{supp(w)\cap\Delta_{\mathcal P}}|\nabla w|^2 dx \geq C\, |supp(w)\cap\Delta_{\mathcal P}|^{-2/N} \int_{supp(w)\cap\Delta_{\mathcal P}}w^2 dx
\end{equation}
so long as $|supp(w)\cap\Delta_{\mathcal P}|\leq \eta_0$.
\end{Lemma}

The above lemma follows from Proposition 2.3 (2) of \cite{BV}. We state this proposition for the special case where $w$ belongs to $H^1\cap C$. 

\begin{Prop}[\cite{BV}]\label{prop:poincare}
Let $\widehat\Omega$ be a domain in $\R^N$, not necessarily bounded, with a uniformly Lipschitz boundary, and let $\eta_0$ be a real number with $0<\eta_0<|\widehat\Omega|$. Then there exists a constant $C$ depending only on $\widehat\Omega$ and $\eta_0$ such that, for any open set $D\subset\widehat\Omega$ satisfying $|D|\leq\eta_0$ and any function $w\in H^1(\widehat\Omega)\cap C(\widehat\Omega)$ such that $w=0$ in $\widehat\Omega\setminus D$, the following inequality holds:
\begin{equation}\label{sobolev-general}
\begin{split}
\int_{D}|\nabla w|^2 dx \geq C\left(\int_{D} |w|^{2N/(N-2)} dx\right)^{(N-2)/N} \ \ & \hbox{if}\ \ N\geq 3,\\
\int_{D}|\nabla w|^2 dx \geq C \left(\int_{D}|w|^{\gamma-2}dx\right)^{-1}\int_{D}|w|^\gamma dx\ \ \ & \hbox{if}\ \ N=2,
\end{split}
\end{equation}
where $\gamma\ge2$ is arbitrary if $N=2$.
\end{Prop}

By H${\rm \ddot{o}}$lder's inequality, 
\[
\int_{D} w^2 dx \leq \left(\int_{D} |w|^{2N/(N-2)} dx\right)^{(N-2)/N}\left(\int_D dx\right)^{2/N}.
\]
Combining this with \eqref{sobolev-general} for $N\geq 3$, we obtain
\begin{equation}\label{poincare-general}
\int_{D}|\nabla w|^2 dx \geq C |D|^{-2/N}\int_{D} w^2 dx.
\end{equation}
In the case $N=2$, the above inequality follows by setting $\gamma=2$ in \eqref{sobolev-general}.

An important point of the estimate \eqref{poincare-general} is that, unlike the standard Poincar\'e inequality, $w$ is required to be $0$ only on $\partial D\cap\widehat\Omega$ and no restriction is imposed on the value of $w$ on~$\partial D\cap\partial\widehat\Omega$. This aspect of estimate \eqref{poincare-general} will be important in the proof of Lemma~\ref{lem:poincare-a}.

\begin{proof}[Proof of Lemma~\ref{lem:poincare-a}]
We divide the integral on the left-hand side of \eqref{poincare-a} as follows:
\[
I_{-}=\int_{supp(w)\cap \widehat\Omega_{-}}|\nabla w|^2 dx,\quad
I_{0}=\int_{supp(w)\cap \widehat\Omega^{\ep}_{0}}|\nabla w|^2 dx,\quad
I_{+}=\int_{supp(w)\cap \widehat\Omega_{+}}|\nabla w|^2 dx,
\]
where
\[
\widehat\Omega_{-}=\Delta_{\mathcal P}\cap\{x_1<0\},\quad
\widehat\Omega_{+}=\Delta_{\mathcal P}\cap\{x_1>M\},\ \ 
\]
\[
\widehat\Omega^{\,\ep}_{0}=\Delta_{\mathcal P}\cap\big((0,M)\times \left(\R^{N-1}_y\setminus{\mathcal N}_\ep(\Sigma)\right)\big).
\]
These are all domains with uniformly Lipschitz boundaries, with $\widehat\Omega_{\pm}$ being unbounded, while $\widehat\Omega^{\,\ep}_{0}$ is bounded and $\ep$-dependent.  Note also that $D^\lambda$ is bounded, since $D^\lambda\subset\overline{R^\lambda}$. 

We first consider the integral $I_0$. This integral is taken over the region $supp(w)\cap\hspace{1pt} \widehat\Omega^{\,\ep}_{0}$. The set $\widehat\Omega^{\,\ep}_{0}$ is a close approximation of $\widehat\Omega^{\,0}_{0}:=\Delta_{\mathcal P}\cap\left((0,M)\times \left(\R^{N-1}_y\setminus\Sigma\right)\right)$, which is a bounded open set having finitely many connected components each of which being a bounded domain with a Lipschitz boundary, by virtue of the assumption on $\Sigma$ given in subsection~\ref{ss:main-parallel-blade}. Each connected component of $\widehat\Omega^{\,\ep}_{0}$ is also a bounded domain with a Lipschitz boundary and is a close approximation of the corresponding connected component of $\widehat\Omega^{\,0}_{0}$. It is not difficult to see that there is a diffeomorphism between each connected component of $\widehat\Omega^{\,\ep}_{0}$ and the corresponding connected component of $\widehat\Omega^{\,0}_{0}$ whose Jacobian matrix is uniformly close to identity for all sufficiently small $\ep>0$. Thus the estimate \eqref{poincare-general} applies to each of the connected components of $\widehat\Omega^{\,\ep}_{0}$ with a constant $C=C_0$ that is independent of $\ep$. Consequently, by Proposition~\ref{prop:poincare} and \eqref{poincare-general}, if $\eta_0$ is chosen relatively small, we have
\begin{equation}\label{I0}
I_0\geq C_0 |supp(w)\cap\Delta_{\mathcal P}|^{-2/N}\int_{supp(w)\cap \widehat\Omega^{\ep}_{0}} w^2 dx,
\end{equation}
so long as $|supp(w)\cap\widehat\Omega^{\ep}_{0}|\leq \eta_0$ and $|supp(w)\cap\Delta_{\mathcal P}|\ne 0$.

Next, as regards the integral $I_{-}$ and $I_{+}$, since $\widehat\Omega_-$ and $\widehat\Omega_+$ are simple domains that do not depend on $\ep$, we can apply Proposition~\ref{prop:poincare} and \eqref{poincare-general} directly, to obtain
\begin{equation}\label{I+-}
I_{\pm}\geq C_\pm |supp(w)\cap\Delta_{\mathcal P}|^{-2/N}\int_{supp(w)\cap \widehat\Omega_{\pm}} w^2 dx,
\end{equation}
for some constant $C\pm>0$, provided that $|supp(w)\cap\widehat\Omega^{\pm}|\leq \eta_0$ and $|supp(w)\cap\Delta_{\mathcal P}|\ne 0$. Combining \eqref{I0} and \eqref{I+-}, we obtain \eqref{poincare-a}. The lemma is proved.
\end{proof}


\end{document}